\documentclass[11pt,letterpaper]{amsart}

\usepackage[english]{babel}

\usepackage[T1]{fontenc}
\usepackage{lmodern}

\usepackage{amsmath,amsthm,amssymb,amsfonts}%AMS PACKAGE
\usepackage{mathtools}
\usepackage{tikz-cd}%adds diagrams
\usepackage{mathrsfs}%gives fancy mathscr
\usepackage{thmtools}
\usepackage{slashed}
\usepackage{mathabx}

\usepackage{bbm}
\usepackage{microtype}
\usepackage{accents}

\usepackage{enumitem}%enumerate package
\usepackage[toc,page]{appendix}%adds appendix
\usepackage{imakeidx}%allows one to maked indices in articles

\usepackage[backend=biber,style=alphabetic,sorting=ynt,sortcites]{biblatex}%References
\allowdisplaybreaks[1]

\theoremstyle{plain}
\newtheorem{theorem}{Theorem}[section]
\newtheorem*{theorem*}{Theorem}
\newtheorem{thmx}{Theorem}

\newtheorem{lem}[theorem]{Lemma}
\newtheorem{prop}[theorem]{Proposition}
\theoremstyle{definition}
\newtheorem{ex}[theorem]{Example}

\newtheorem{dfn}[theorem]{Definition}
\newtheorem{rem}[theorem]{Remark}

\theoremstyle{remark}

\newcommand{\RNum}[1]{\uppercase\expandafter{\romannumeral #1\relax}}

\makeatletter
\providecommand*{\twoheadrightarrowfill@}{%
  \arrowfill@\relbar\relbar\twoheadrightarrow
}
\providecommand*{\twoheadleftarrowfill@}{%
  \arrowfill@\twoheadleftarrow\relbar\relbar
}
\providecommand*{\xtwoheadrightarrow}[2][]{%
  \ext@arrow 0579\twoheadrightarrowfill@{#1}{#2}%
}
\providecommand*{\xtwoheadleftarrow}[2][]{%
  \ext@arrow 5097\twoheadleftarrowfill@{#1}{#2}%
}
\newcommand\setItemnumber[1]{\setcounter{enum\romannumeral\@enumdepth}{\numexpr#1-1\relax}}
\makeatother

\newcommand\norm[1]{\left\lVert#1\right\rVert}

\DeclareMathOperator\ev{ev}

\DeclareMathOperator\im{Im}

\DeclareMathOperator\dom{Dom}

\newcommand{\R}{\mathbb{R}}

\newcommand{\N}{\mathbb{N}}
\newcommand{\C}{\mathbb{C}}

\newcommand{\cA}{\mathcal{A}}

\newcommand{\cF}{\mathcal{F}}
\newcommand{\cG}{\mathcal{G}}

\newcommand{\cQ}{\mathcal{Q}}

\newcommand{\cX}{\mathcal{X}}

\newcommand{\Rpt}{\mathbb{R}_+^\times}

\usetikzlibrary{decorations.markings}
\AtEveryBibitem{\clearfield{url}}
\newcommand{\gr}{\mathfrak{gr}}
\newcommand{\Gr}{\mathfrak{gr}}
\newcommand{\aF}{\mathfrak{a}\cF}

\newcommand{\Grass}{\mathrm{Grass}}

\newcommand\norml[1]{{\left\vert\kern-0.25ex\left\vert\kern-0.25ex\left\vert #1 
    \right\vert\kern-0.25ex\right\vert\kern-0.25ex\right\vert}}

\usepackage{mathbbol}
\DeclareSymbolFontAlphabet{\mathbb}{AMSb}
\DeclareSymbolFontAlphabet{\mathbbl}{bbold}

\begin{document}
\title{Tangent groupoid and tangent cones in sub-Riemannian geometry}
\author{Omar Mohsen}
\date{}
\begin{abstract}
  Let $X_1,\cdots,X_m$ be vector fields satisfying Hörmander's Lie bracket generating condition on a smooth manifold $M$. We generalise Connes's tangent groupoid, by constructing a completion of the space $M\times M\times \R_+^\times$ using the sub-Riemannian metric. We use our space to calculate all the tangent cones  of the sub-Riemannian metric in the sense of the Gromov-Hausdorff distance. This generalises a result of Bellaïche.
  \end{abstract}
\maketitle

\setcounter{tocdepth}{1}
\tableofcontents
\section*{Introduction}
Let $M$ be a smooth manifold. In \cite[Chapter 2.5]{ConnesBook}, Connes introduced the tangent groupoid $$\mathbb{T}M:=M\times M\times \R_+^\times \sqcup TM\times \{0\}.$$ 
The space $\mathbb{T}M$ is equipped with a smooth manifold structure by the deformation to the normal cone construction, see \cite[Chapter 5]{FultonBookIntersection}. 
The space $\mathbb{T}M$ is also a Lie groupoid, in particular the space $C^\infty_c(\mathbb{T}M)$ is equipped with a convolution law making it a $*$-algebra. The convolution law is a deformation of the usual convolution law on $M\times M$ which is used in Schwartz kernels of operators acting on $L^2M$ and the convolution law on $T_xM$ for $x\in M$ which comes from the commutative group structure. The tangent groupoid has been studied and generalised by many authors, see \cite{Erik1,HigsonYiSpinors,ConnesTangentStrictQuant,PaoloSchwartzAlgebraTangent,NistorPseudoContFamily,AS3,MontPierrotIndiceAna,DebordLescureKDual,HilsumSkandalisFoli,HigsonTnagentGrp,HigsonHaj,MohsenGrpTangent,ErikBobTangentGrp,ChoiPonge,LandsmanRamazanContinuousField,EwertSchwartz}.\footnote{This list is certainly not exhaustive.}

Following Gromov \cite{GromovSubRiem}, if $(X,d)$ is a metric space, then the limit (if it exists) $\lim_{t\to 0^+}(X,t^{-1}d,x)$ in the sense of pointed Gromov-Hausdorff distance is called the tangent cone of $X$ at $x$. The tangent cone of a Riemannian manifold $(M,d_M)$ at $x$ is $T_xM$. Connes tangent groupoid can be seen as an improvement of this simple fact, because the function \begin{equation}\label{eqn:dist_Connes}
\mathbb{T}M\to \R_+,\quad (y,x,t)\mapsto t^{-1}d_M(x,y),\quad  (\xi,x,0)\mapsto d_{T_xM}(\xi,0)
\end{equation}is continuous, where $d_{T_xM}$ is the constant distance associated to the Riemannian metric. 

In this article, we present a generalization of Connes's tangent groupoid in sub-Riemannian geometry. This construction is a special case of the author's construction \cite{MohsenBlowup}, and has been used in \cite{MohsenMaxHypo,MohsenIndex}. The construction given in \cite{MohsenBlowup} is based on ideas and techniques due to Androulidakis and Skandalis \cite{AS1,AS2} and Debord \cite{DebordFoliation2001,Debord2013}. The goal here is to give a self-contained presentation of our construction which does not rely on singular foliations. We then use our construction to calculate the tangent cones in sub-Riemannian geometry.

So let $X_1,\cdots,X_d$ be vector fields which satisfy Hörmander's Lie bracket generating condition of depth $N\in \N$, i.e., for every $x\in M$, the tangent space $T_xM$ is spanned by $X_1(x),\cdots,X_d(x)$ together with their iterated brackets $[X_i,X_j](x),[[X_i,X_j],X_k](x),\cdots$ of length $\leq N$. A theorem of Chow \cite{ChowTheorem} implies that any two points of $M$ can be connected by a piecewise smooth path $\gamma$ such that $\gamma'(t)\in \mathrm{span}(X_1(\gamma(t)),\cdots,X_d(\gamma(t)))$ for almost all $t\in [0,1]$. One then defines the Carnot-Carathéodory metric $d_{CC}(x,y)$ to be the minimum length of all such paths, see \cite{BellaicheArt}.

Generalizing a result of Mitchell \cite{MitchellCarnot}, Bellaïche \cite{BellaicheArt} computed the tangent cone of $(M,d_{CC})$ as follows. Let $G$ be the free $N$-step nilpotent Lie group with $d$ generators, one for each $X_1,\cdots,X_d$. For each $x\in M$, 
Bellaïche identifies a simply connected Lie subgroup $\mathfrak{r}_x\subseteq G$ of codimension $\dim(M)$. He shows that \begin{equation}\label{eqn:limit_bellaiche}
\lim_{t\to 0^+}(M,t^{-1}d_{CC},x)=(G/\mathfrak{r}_x,d_{G/\mathfrak{r}_x},\mathfrak{r}_x)
\end{equation}
where $d_{G/\mathfrak{r}_x}$ is a Carnot-Carathéodory metric on the homogeneous space $G/\mathfrak{r}_x$.

A naive guess of the analogue of Connes's tangent groupoid is $M\times M\times \R_+^\times\bigsqcup_{x\in M}G/\mathfrak{r}_x\times \{0\}$. There are two problems with this definition. First it is not a groupoid because $\mathfrak{r}_x$ in general is not normal in $G$. 
Second it is not locally compact (with respect to some natural topology one can define). These issues are ultimately due to the fact that the limit \eqref{eqn:limit_bellaiche} is not uniform in $x$. That is, if $x_n\in M$ is a sequence converging to $x$ and $t_n\to 0^+$, then $\lim_{n\to +\infty}(M,t_n^{-1}d_{CC},x_n)$ exists (up to taking a subsequence) but the limit is of the form $(G/H,d_{G/H},H)$ where $H\subseteq G$ is a simply connected Lie group and in general $H\neq \mathfrak{r}_y$ for all $y\in M$.

This leads to our construction. For each $x\in M$, we define a subset $\cG^0_x$ of simply connected Lie subgroups of $G$ of codimension $\dim(M)$, which satisfies the following conditions \begin{enumerate}
\item $\mathfrak{r}_x\in \cG^0_x$.
\item If $g\in G$, $H\in \cG^0_x$, then $gHg^{-1}\in \cG^0_x$.
\item The set $\cG^0_x$ is closed when identified (using Lie algebras) with a subset of the Grassmannian manifold of subspaces of $\mathfrak{g}$ of codimension $\dim(M)$. 
\end{enumerate}

The set $\cG^0_x$ comes from looking at linear relations between the vector fields $X_1,\cdots,X_d$ and their commutators which we now define.
Let $\mathfrak{g}$ be the Lie algebra of $G$. There is a natural linear map $$\natural:\mathfrak{g}\to \cX(M),$$ where $\cX(M)$ is the space of vector fields on $M$. The map $\natural$ is the unique linear map which sends $X_i$ to $X_i$ and which sends iterated Lie brackets 
of $X_i$ in $\mathfrak{g}$ of length $\leq N$ to the corresponding Lie brackets in $\cX(M)$. Notice that $\natural$ isn't a Lie algebra homomorphism because iterated Lie brackets in $\cX(M)$ of length $> N$ do not necessarily vanish.
For $x\in M$, let $\natural_x:\mathfrak{g}\to T_xM$ be the composition of $\natural$ with the evaluation map at $x\in M$. The map $\natural_x$ is surjective by Hörmander's condition. 
Therefore, $\ker(\natural_x)$ is a subspace of $\mathfrak{g}$ of codimension equal to $\dim(M)$. Let $\mathrm{Grass}(\mathfrak{g})$ be the Grassmannian manifold of subspaces of codimension $\dim(M)$. 
We also make use of the natural graded dilation $\alpha_t:\mathfrak{g}\to \mathfrak{g}$ on $\mathfrak{g}$ which is defined by $\alpha_t(X_i)=tX_i$, $\alpha_t([X_i,X_j])=t^2[X_i,X_j]$, etc. 
The main point of our construction is that the convergence in the Gromov-Hausdorff distance of $(M,t_n^{-1}d_{CC},x_n)$ can be reduced to the linear algebra problem of computing the limit of $\alpha_{t_n^{-1}}(\ker(\natural_{x_n}))$ in the Grassmannian manifold $\mathrm{Grass}(\mathfrak{g})$.
We thus define 
    \begin{equation*}\begin{aligned}
        \cG^0_x:=\{\mathfrak{h}\in\mathrm{Grass}(\mathfrak{g}) :\exists (t_n)_{n\in \N}\subseteq \Rpt,(x_n)_{n\in \N}\subseteq M \text{ such that }\\
        t_n\to 0,x_n\to x,\alpha_{t_n^{-1}}(\ker(\natural_{x_n}))\to \mathfrak{h}\}
    \end{aligned}\end{equation*}
Even though $\natural$ is not a Lie algebra homomorphism and therefore $\ker(\natural_x)$ is not a Lie subalgebra of $\mathfrak{g}$ in general. We prove that the limit $\alpha_{t_n^{-1}}(\ker(\natural_{x_n}))$, if it exists, is always a Lie subalgebra of $\mathfrak{g}$. 
Therefore, $\cG^0_x$ can be considered as a set of Lie subgroups of $G$.

We go further by creating a space that includes all the spaces $\cG^0_x$ together as $x$ varies. To this end, we consider the inclusion 
    \begin{equation*}\begin{aligned}
        M\times \Rpt\to \mathrm{Grass}(\mathfrak{g})\times M\times \Rpt\\
        (x,t)\mapsto (\alpha_{\frac{1}{t}}(\ker(\natural_x)),x,t)
    \end{aligned}\end{equation*}
We take the space $$\mathbb{G}^0:=M\times \R_+^\times \bigsqcup_{x\in M}\cG^0_x\times \{(x,0)\}$$ to be the closure of $M\times \Rpt$ inside $\mathrm{Grass}(\mathfrak{g})\times M\times \Rpt$ equipped with the subspace topology. This topology is a second countable locally compact metrizable topology
which makes the natural projection $\mathbb{G}^0\to M\times \R$ a proper continuous map. Our main theorem is the following 
\begin{thmx}\label{thm_intro_tangentcones}If $(x_n,t_n)\in M\times \R_+^\times$ converges in $\mathbb{G}^0$ to $(H,x,0)$ where $H\in \cG^0_x$, i.e., if $t_n\to 0$, $x_n\to x$ and $\alpha_{t_n^{-1}}(\ker(\natural_{x_n}))\to \mathfrak{h}$, then $$\lim_{n\to +\infty}(M,t_n^{-1}d_{CC},x_n)=(G/H,d_{G/H},H),$$ where $d_{G/H}$ is a Carnot-Carathéodory metric on the homogeneous space $G/H$.
\end{thmx}
Since the projection map $\mathbb{G}^0\to M\times \R$ is proper, it follows that Theorem \ref{thm_intro_tangentcones} gives all possible tangent cones.

We prove Theorem \ref{thm_intro_tangentcones} by constructing the analogue of Connes's tangent groupoid which as a set is equal to $$\mathbb{G}:=M\times M\times \R_+^\times\bigsqcup_{x\in M,H\in \cG^0_x}G/H\times \{0\}\rightrightarrows \mathbb{G}^0.$$ 
We equip $\mathbb{G}$ with a second countable locally compact metrizable topology. We show that $\mathbb{G}$ is a continuous family groupoid in the sense of Paterson \cite{ContFamilyGroupoids}. 
Finally, to deduce Theorem \ref{thm_intro_tangentcones}, we prove that the following function $$\mathbb{G}\to \R_+,\quad  (y,x,t)\mapsto t^{-1}d_{CC}(y,x),\quad  (gH,x,0)\mapsto d_{G/H}(gH,H)$$ is continuous. We end the introduction with a few remarks \begin{enumerate}
\item If $X_1,\cdots,X_d$ have constant rank and $X_1,\cdots,X_d,[X_i,X_j]$ have constant rank, and $X_1,\cdots,X_d,[X_i,X_j],[[X_i,X_j],X_k]$ have constant rank, etc., then the limit \eqref{eqn:limit_bellaiche} is uniform in $x$ and $\cG^0_x=\{\mathfrak{r}_x\}$. 
In this case our tangent groupoid coincides with previous constructions by Choi and Ponge \cite{ChoiPonge} and van Erp and Yuncken \cite{ErikBobTangentGrp}, see also \cite{MohsenGrpTangent, HigsonHaj} for different constructions of the tangent groupoid in this case.
\item The above strategy of computing tangent cones by defining a deformation groupoid appears in Pansu's work \cite{PansuGroupod}, where he computes the tangent cones at infinity of simply connected nilpotent Lie groups.
\item The space $\mathbb{G}^0$ is a closed subset of $\mathbb{G}$. There is a natural $\R_+^\times$-action on $\mathbb{G}$ which leaves the set $\mathbb{G}^0$ invariant. We show that the $\Rpt$-action is proper on $\mathbb{G}\backslash \mathbb{G}^0$ and that the quotient $\mathbb{P}(\mathbb{G}):=(\mathbb{G}\backslash \mathbb{G}^0)/\Rpt$ is second countable metrizable space which is compact if $M$ is compact. As a set $\mathbb{P}(\mathbb{G})$ is equal to the union of the dense open subset $(M\times M)\backslash \Delta_M$ and some other set coming from the tangent cones. Hence, we can view $\mathbb{P}(\mathbb{G})$ as a blowup of $M\times M$ along the diagonal using the sub-Riemannian metric. In a future paper, we will show the connection between this space and the heat kernel and Weyl law of maximally hypoelliptic differential operators, see Remark \ref{rem:heat_ker}.
\end{enumerate}

\textbf{Structure of the paper.}
We think it is easier for the reader to go through the construction of $\mathbb{G}$ and $\mathbb{G}^0$ first, then go through the proofs of their properties. For this reason, the proofs of the results in Sections \ref{sec:tangent_cones} and \ref{sec:tangent_groupoid} are given in Section \ref{sec:proofs}.

The article is organized as follows \begin{itemize}
\item In Section \ref{sec:tangent_cones}, we define the space $\mathbb{G}^0$.
\item In Section \ref{sec:tangent_groupoid}, we define the space $\mathbb{G}$.
\item In Section \ref{sec:proofs}, we prove the results stated in Sections \ref{sec:tangent_cones} and \ref{sec:tangent_groupoid}.
\item In Section \ref{sec:cont_fam_grp}, we prove that $\mathbb{G}$ is a continuous family groupoid in the sense of  Paterson {\cite{ContFamilyGroupoids}}.
\item In Section \ref{sec:Lie_Alg}, we define the Lie algebroid of $\mathbb{G}$.
\item In Section \ref{sec:Debord-Skand-Act}, we define a natural $\R_+^\times$-action. It is the generalization of the action introduced by Debord and Skandalis \cite{DebordSkandalis1} on Connes's tangent groupoid.
\item In Section \ref{sec:blowup_Riem}, we construct the quotient space $\mathbb{P}(\mathbb{G})$.
\item In Section \ref{sec:TangentCones}, we prove Theorem \ref{thm_intro_tangentcones}.
\end{itemize}

\textbf{Acknowledgements.}
We thank P. Pansu for interesting discussions which led to the formulation of Theorem \ref{thm_intro_tangentcones}.	We thank the referees for their remarks and suggestions which helped improve the article.

\textbf{Notations and Conventions.}
\begin{itemize}
\item We use $\exp(X)\cdot x$ for the flow of $X\in \cX(M)$ at time $1$ starting from $x\in M$. Sometimes $X$ will depend on some parameters, denoted by $t$. We stress that throughout the article, we never take any time dependent flows.
\item If $G$ is a Lie group, then its Lie algebra is the vector space of \textit{right} invariant vector fields on $G$. This differs by a sign from the convention usually used in Lie group theory but agrees with the convention usually used in Lie groupoid theory. In particular, for the sake of completeness, the Baker-Campbell-Hausdorff formula is \begin{equation}\label{eqn:Camp}
  \log(e^v\cdot e^w)=v+w-\frac{1}{2}[v,w]+\frac{1}{12}[v,[v,w]+\cdots,\quad v,w\in \mathfrak{g}.
\end{equation}
If $\mathfrak{g}$ is a nilpotent Lie algebra, then the BCH formula is a finite sum, and hence equips $\mathfrak{g}$ with the structure of a Lie group. We will use $\mathfrak{g}$ to denote both the Lie algebra and the Lie group, usually denoting the product by $v\cdot w$ for $v,w\in \mathfrak{g}$. We will usually use $v^{-1}$ instead of $-v$ to denote the inverse of $v$.

If $L\subseteq \mathfrak{g}$ is a Lie subalgebra, then it is also a Lie subgroup. We will use $\mathfrak{g}/L$ to denote the set of left cosets of $L$ and $\frac{\mathfrak{g}}{L}$ for the (rarely used) quotient of $\mathfrak{g}$ by $L$ as vector spaces.
\item If $X,Y$ are topological spaces, then a partially defined function $f:X\to Y$ will always mean a function $f$ defined on an open subset of $X$. 
\end{itemize} 
\section{Tangent cones}\label{sec:tangent_cones}
 \begin{dfn}[\cite{MohsenMaxHypo}] Let $M$ be a smooth manifold.  We recall that $\cX(M)$ denotes the Lie algebra of vector fields on $M$. A weighted sub-Riemannian structure of depth $N\in \N$ is a family $$0=\cF^0\subseteq \cF^1\subseteq \cdots\subseteq \cF^N=\cX(M)$$ of finitely generated $C^\infty(M,\R)$-modules of vector fields such that for all $i,j\in \N$ \begin{equation}\label{eqn:BracketF}
 [\cF^i,\cF^j]\subseteq \cF^{i+j}.
\end{equation}
We will use the convention $\cF^i=\cF^N$ for all $i\geq N$ throughout the article.\end{dfn}
\begin{rem} The results of this article can be extended to the more general situation where $\cF^i$ are only locally finitely generated. We will restrict to the finitely generated case for simplicity.
\end{rem}
 \begin{ex}\label{exs:filtrations} Let $X_1,\cdots,X_d\in \cX(M)$ be vector fields satisfying Hörmander's Lie bracket generating condition of depth $N\in \N$. We can define $\cF^\bullet$ inductively by \begin{align*}
\cF^1=\langle X_1,\cdots,X_d\rangle,\quad \cF^k=\cF^{k-1}+[\cF^{1},\cF^{k-1}].
\end{align*}
More generally if $v_1,\cdots,v_d\in \N$ are natural numbers, thought of as weights of $X_1,\cdots,X_d$. Then we can define $\cF^k$ to be the module generated by the iterated Lie brackets $$[X_{a_1},[\cdots,[X_{a_{n-1}},X_{a_n}]\cdots]$$such that $\sum_{i=1}^nv_{a_i}\leq k$. Hörmander's condition implies that $\cF^{N\max(v_1,\cdots,v_d)}=\cX(M)$.
\end{ex}
If
\begin{equation}\label{eqn:sub-riem}
  \cF^{i+1}=\cF^{i}+[\cF^{i},\cF^{1}]
 \end{equation} for all $i\geq 1$, then a weighted sub-Riemannian structure is just a sub-Riemannian structure. By this, we mean that $\cF^1$ is a module generated by some vector fields which satisfy Hörmander's condition of depth $N$, and $\cF^i$ is the module generated by their iterated Lie brackets of length $\leq i$. 
 
\textbf{Osculating groups.} Let $x\in M$. We define \begin{equation}\label{eqn:dfn_gr}
 \gr(\cF)_x=\bigoplus_{i=1}^N\frac{\cF^i}{\cF^{i-1}+I_x\cF^i},
\end{equation} where $I_x\subseteq C^\infty(M,\R)$ is the ideal of smooth functions vanishing at $x$.  If $X\in \cF^i$, then $[X]_{i,x}$ denotes the class of $X$ in $\frac{\cF^i}{\cF^{i-1}+I_x\cF^i}\subseteq \gr(\cF)_x$. By \eqref{eqn:BracketF}, the Lie bracket of vector fields descends to a bilinear map \begin{align*}\frac{\cF^i}{\cF^{i-1}+I_x\cF^i}\times \frac{\cF^j}{\cF^{j-1}+I_x\cF^j}&\to \frac{\cF^{i+j}}{\cF^{i+j-1}+I_x\cF^{i+j}}.
\end{align*} The space $\gr(\cF)_x$ is thus equipped with the structure of a graded nilpotent Lie algebra given by \begin{align*}
   \forall X\in \cF^i,Y\in \cF^j,\quad [[X]_{i,x},[Y]_{j,x}]:=\begin{cases} \Big[[X,Y]\Big]_{i+j,x}\quad &\text{if}\ i+j\leq N\\0\quad &\text{if}\ i+j>N
   \end{cases}
\end{align*}Hence $\gr(\cF)_x$ is a Lie group by the BCH formula \eqref{eqn:Camp}.

\textbf{Space of tangent cones.}
 For $x\in M$, we denote by $\Grass(\cF)_x$ the Grassmannian manifold of linear subspaces of $\gr(\cF)_x$ of codimension equal to $\dim(M)$. Let  \begin{align}\label{eqn:Grassaf*}
  \Grass(\aF):=M\times \R_+^\times\bigsqcup_{x\in M} \Grass(\cF)_x\times \{(x,0)\}.
\end{align}
We will equip $\Grass(\aF)$ with a topology. To do so we need to choose a set of generators for each module $\cF^i$. We formalise this as follows \begin{dfn}\label{rem:simpler_basis}A graded basis is a pair $(V,\natural)$ where $V=\bigoplus_{i=1}^NV^i$ is a graded $\R$-vector space, \begin{equation}\label{eqn:basis}
 \natural:V\to \cX(M)
\end{equation} an $\R$-linear map 
such that \begin{enumerate}
\item for all $k\in \{1,\cdots,N\}$, $\natural(V^k)\subseteq \cF^k$.
\item  for all $k\in \{1,\cdots,N\}$, the $C^\infty(M,\R)$ module generated by $\natural(\bigoplus_{i=1}^kV^i)$ is equal to $\cF^k$.
\end{enumerate}\end{dfn}

Let $(V,\natural)$ be a graded basis. For $(x,t)\in M\times \R_+^\times$, we define the linear maps
\begin{equation}\label{eqn:natural_V}
\begin{aligned}
&\natural_{t}:V\to \cX(M),\quad \natural_{t}\left(\sum_{i=1}^Nv_i\right)=\sum_{i=1}^Nt^i\natural(v_i),\quad v_i\in V^i\\
& \natural_{x,t}:V\to T_xM,\quad \natural_{x,t}(v)=\natural_{t}(v)(x),\quad v\in V.
\end{aligned}
\end{equation} The map $\natural_{x,t}$ is surjective because $\natural(V)$ generates $\cX(M)$. We define a grading preserving linear map $$\natural_{x,0}:V\to \gr(\cF)_x,\quad \sum_{i=1}^Nv_i\mapsto \sum_{i=1}^N[\natural(v_i)]_{i,x},\quad v_i\in V^i.$$The map $\natural_{x,0}$ is also surjective, see \cite[Proposition 1.1]{MohsenMaxHypo}.

Let $\Grass(V)$ be the Grassmannian manifold of linear subspaces of $V$ of codimension equal to $\dim(M)$. We define \begin{equation}\label{eqn:Grass_natural}
\begin{aligned}
\Grass(\natural):\Grass(\aF)&\to  \Grass(V)\times M\times \R_+\\
(x,t)&\mapsto (\ker(\natural_{(x,t)}),x,t),\quad t\neq 0\\
(L,x,0)&\mapsto (\natural_{x,0}^{-1}(L),x,0),\quad L\in \Grass(\cF)_x
\end{aligned}
\end{equation}
The map $\Grass(\natural)$ is clearly injective. Thus, we are lead to the following definition.
\begin{dfn}The topology on $\Grass(\aF)$ is the topology which makes $\Grass(\natural)$ an embedding.
\end{dfn}
We have used the graded basis $(V,\natural)$ to define the topology on $\Grass(\aF)$. The following proposition which will be proved in Section \ref{sec:proof_prop_well_def_top} says that if we pick a different graded basis, then we would obtain the same topology.
\begin{prop}\label{thm:top_well_defined}\begin{enumerate}
\item The image of the map $\Grass(\natural)$ is a closed subset of $\Grass(V)\times M\times \R_+$.
\item The space $\Grass(\aF)$ is second countable locally compact metrizable, and the natural projection $\Grass(\aF)\to M\times \R_+$ is a continuous proper map. 
\item  The topology on $\Grass(\aF)$ does not depend on the choice of the graded basis. 
\end{enumerate}
\end{prop}
In general $M\times \R_+^\times$ is not dense in $\Grass(\aF)$. This leads to the following 
\begin{dfn}\label{dfn:characteristic_set}\label{dfn:bG0} We define \begin{align*}\mathcal{G}^0_x:&=\{L\in \Grass(\cF)_x:(L,x,0)\in \overline{M\times \R_+^\times}\subseteq \Grass(\aF) \}\\
   \mathcal{G}^0:&=\{(L,x):x\in M,L\in \Grass(\cF)_x\}\\
    \mathbb{G}^0:&=\overline{M\times \R_+^\times}=M\times \R_+^\times\sqcup \mathcal{G}^0\times \{0\}\subseteq \Grass(\aF).
\end{align*}
\end{dfn}
We equip these spaces with the subspace topology from $\Grass(\aF)$. We stress that for the rest of the article, we will not use the space $\Grass(\aF)$. Only the subspace $\mathbb{G}^0$ is of interest to us. 
\begin{theorem}\label{thm:prop of blowup}
Let $x\in M$. The following holds
\begin{enumerate}
\item If $L\in \mathcal{G}^0_x$, then $L$ is a Lie subalgebra and a Lie subgroup of $\gr(\cF)_x$.
\item If $g\in \Gr(\cF)_x$, $L\in \mathcal{G}^0_x$, then $gLg^{-1}\in \mathcal{G}^0_x$.
\end{enumerate}
\end{theorem}
Theorem \ref{thm:prop of blowup} will be proved in Section \ref{sec:proof_of_prop_of_blowup}.
 \begin{ex}[Generalised Grushin plane]\label{ex:mainex}Consider $M=\R^2$, $N\in \N$, with $N>1$. Let $$\cF^1=\langle \partial_x,x^{N-1}\partial_y\rangle,\cF^2=\langle \partial_x,x^{N-2}\partial_y\rangle,\cdots,\cF^{N}=\langle \partial_x,\partial_y\rangle=\cX(\R^2).$$ This is the filtration associated to the vector fields $\partial_x$ and $x^{N-1}\partial_y$. A natural graded basis is $V\subseteq \cX(M)$ the linear span of $\partial_x,x^{N-1}\partial_y,\cdots,\partial_y$ with $\natural$ the inclusion map. The grading on $V$ is given by declaring $\partial_x$ to be of degree $1$ and $x^{i}\partial_y$ to be of degree $N-i$. A straightforward computation shows that \begin{align*}
\gr(\cF)_{(a,b)}=\begin{cases}\R[\partial_x]_{1,(a,b)}\oplus \R[x^{N-1}\partial_y]_{1,(a,b)},\quad &a\neq 0\\ \R[\partial_x]_{1,(0,b)}\oplus \R[x^{N-1}\partial_y]_{1,(0,b)}\oplus \cdots \oplus \R[\partial_y]_{N,(0,b)},\quad &a=0\end{cases}.
\end{align*}
The non-trivial Lie brackets are given by $$\big[[\partial_x]_{1,(0,b)},[x^i\partial_y]_{N-i,(0,b)}\big]= i[x^{i-1}\partial_y]_{N-i+1,(0,b)},\quad i\geq 1.$$ If $a\neq 0$, then $\cG^{0}_{(a,b)}=\{(0)\}$ as this is the only subspace of codimension $2$. A simple computation shows that the space $\cG^{0}_{(0,b)}$ consists of the subspaces 
\begin{equation}\label{eqn:G0y_in_examp}
 \Bigg\{\mathrm{span}\Big(\lambda[x^{i-1}\partial_y]_{N-i+1,(0,b)}-[x^{i}\partial_y]_{N-i,(0,b)}:1\leq i\leq N-1\Big):\lambda\in \R\Bigg\}
\end{equation} and \begin{equation}\label{eqn:G0y_in_examp_inf}
 \mathrm{span}\Big([x^i\partial_y]_{N-i,(0,b)}:0\leq i\leq N-2\Big).
\end{equation}
In fact a sequence $((a_n,b_n),t_n)$ converges to $(L,(0,b),0)$ in $\mathbb{G}^0$ for $L \in \cG^{0}_{(0,b)}$ if and only if $(a_n,b_n)\to (0,b),t_n\to 0$ and $\frac{a_n}{t_n}\to \lambda$ if $L$ corresponds to $\lambda$ in \eqref{eqn:G0y_in_examp} or $\frac{a_n}{t_n}\to \infty$ if $L$ is the subspace in \eqref{eqn:G0y_in_examp_inf}. It is easy to see that the space $\mathbb{G}^0$ is homeomorphic to \begin{align*}
\R^2\times \R_+^\times\sqcup \R^*\times \R\times \{0\}\sqcup\mathbb{P}^1(\R)\times \R\times \{0\}
\end{align*}
the classical blow up of $\R^2\times \R_+$ at the subspace $a=t=0$.
\end{ex}
\textbf{Distinguished tangent cone.} Let $x\in M$. There is a distinguished subspace $\mathfrak{r}_x\in \cG^0_x$, which is defined as follows. Let $\ev_x(\cF^i)=\{X(x)\in T_xM:X\in \cF^i\}$. There is a natural evaluation map \begin{equation}\label{eqn:ev_map}
\ev_x:\gr(\cF)_x\to \bigoplus_{i=1}^N\frac{\ev_x(\cF^i)}{\ev_x(\cF^{i-1})}
\end{equation} which takes an element $[X]_{i,x}$ and maps it to $[X(x)]\in \frac{\ev_x(\cF^i)}{\ev_x(\cF^{i-1})}$ for $X\in \cF^i$. The subspace $\mathfrak{r}_x$ is defined as the kernel of $\ev_x$.
\begin{prop}\label{prop:dist_point}For any $x\in M$ and sequence $(t_n)_{n\in \N}\subseteq \Rpt$, the limit of $(x,t_n)$ in the space $\mathbb{G}^0$, as $n\to +\infty$ is equal to $(\mathfrak{r}_x,x,0)$.
\end{prop}
\begin{proof}
Since the topology of $\Grass(\aF)$ is independent of the choice of graded basis, we are free to choose the graded basis to check convergence. To this end, we choose a basis as follows. Let $X_1,\cdots,X_{k_1}\in \cF^1$ be generators. By constant coefficient linear transformation, we can change the basis so that $X_1(x),\cdots,X_{s_1}(x)$ for some $s_1$ form a basis of $\ev_x(\cF^1)$ and $X_{s_1+1}(x)=\cdots=X_{k_1}(x)=0$. For $\cF^2$, we add to $X_1,\cdots,X_{k_1}$ vector fields $Y_1,\cdots,Y_{k_2}$ such that $X_1,\cdots,X_{k_1},Y_1,\cdots,Y_{k_2}$ generate $\cF^2$. By a constant coefficient linear transformation only changing the $Y$'s, we can suppose that $X_1(x),\cdots,X_{s_1}(x),Y_1(x),\cdots,Y_{s_2}(x)$ form a basis of $\ev(\cF^2)$  and $Y_{s_2+1}(x)=\cdots =Y_{k_2}(x)=0$. We continue this way till $\cF^N$. In the end we get a basis in which $\ker(\natural_{x,t})$ does not depend on $t$ and is equal to $\natural_{x,0}^{-1}(\mathfrak{r}_x)$. The result follows.
\end{proof}

\begin{ex}In Example \ref{ex:mainex}, $\mathfrak{r}_{(0,b)}=\mathrm{span}([x^i\partial_y]_{N-i,(0,b)}:1\leq i\leq N-1)$. We remark that in $\mathbb{G}^0$, the limit of $(\mathfrak{r}_{(a,b)},(a,b),0)=(0,(a,b),0)$ for $a\neq 0$ as $a\to 0$ converges to the point $$\Big(\mathrm{span}([x^i\partial_y]_{N-i,(0,b)}:0\leq i\leq N-2),(0,b),0\Big)$$ which differs from $(\mathfrak{r}_{(0,b)},(0,b),0)$. In particular the map \begin{equation}\label{eqn:mapxtorx}
 M\mapsto \mathbb{G}^0,\quad x\to (\mathfrak{r}_x,x,0)
\end{equation} \textbf{is not} continuous in general. 
\end{ex}
\begin{rem}\label{rem:heat_ker}
The discontinuity of the map \eqref{eqn:mapxtorx} is the reason for the failure of various results in the literature on the heat kernel asymptotics, for example \cite[Theorem A]{ColinLucHeatkernel1}, to be uniform in $x\in M$.
 We will expand further this remark in a future article.
\end{rem}
\section{Tangent groupoid}\label{sec:tangent_groupoid}
By Theorem \ref{thm:prop of blowup}.1, if $L\in \cG^0_x$, then $L$ is a Lie subgroup of $\Gr(\cF)_x$. Let \begin{align*}
\cG_x=\{gL:g\in \Gr(\cF)_x,L\in \cG^0_x\}&=\{Lg:g\in \Gr(\cF)_x,L\in \cG^0_x\}\\&=\{gLh:g,h\in \Gr(\cF)_x,L\in \cG^0_x\}
\end{align*}%I changed here for Duke
 be the set of all cosets of all subgroups in $\cG^0_x$. The fact that left cosets of elements of $\cG^0_x$ are also right cosets follows from Theorem \ref{thm:prop of blowup}.2. Let $$\cG:=\{(gL,x):x\in M,gL\in \cG_x\},\quad  \mathbb{G}:=M\times M\times \R_+^\times\sqcup \cG\times \{0\}.$$
We equip $\mathbb{G}$ with a topology as follows.  A subset $U\subseteq\mathbb{G}$ is open if and only if  \begin{itemize}
\item  $U\cap (M\times M\times \R_+^\times)$ is an open subset of $M\times M\times \R_+^\times$ with its usual topology.
\item  For any graded basis $(V,\natural)$, $\cQ_{V}^{-1}(U)$ is open, where $\cQ_{V}$ is the following partially defined map \begin{equation}\label{eqn:dfn_tangent_grp_flow}
\begin{aligned}
 \cQ_V:V\times \mathbb{G}^0&\to \mathbb{G} \\
 (v,x,t)&\mapsto (\exp(\natural_t(v))\cdot x,x,t),\quad v\in V,x\in M,t\in \R_+^\times\\
 (v,L,x,0)&\mapsto (\natural_{x,0}(v)L,x,0),\quad v\in V,x\in M,L\in \cG^0_x,
\end{aligned}
\end{equation}
The domain of $\cQ_V$ is $$(V\times \cG^0\times \{0\})\sqcup \{(v,x,t)\in V\times M\times \R_+^\times:\exp(\natural_t(v))\cdot x \text{ is well-defined}\}.$$
Since $\natural_{t}(v)\in \cX(M)$ converges to $0$ as $t\to 0^+$ in the $C^\infty$ topology uniformly on compact subsets of $M$, it follows that $\dom(\cQ_V)$ is an open subset of $V\times \mathbb{G}^0$.
\end{itemize}
We equip $\cG$ with the subspace topology from $\mathbb{G}$.
\begin{theorem}\label{thm:top_on_bG}\begin{enumerate}
\item If $U\subseteq \mathbb{G}$ such that $U\cap (M\times M\times \R_+^\times)$ is an open subset of $M\times M\times \R_+^\times$ and $\cQ_{V}^{-1}(U)$ is open for some $(V,\natural)$ graded basis, then $U$ is open.
\item The space $\mathbb{G}$ is a second countable locally compact metrizable topological space. 
\end{enumerate}
\end{theorem}
Convergence in $\mathbb{G}$ is easy to describe as follows 
\begin{prop}\label{prop:Conv}\begin{enumerate}
\item   Let $(y_n,x_n,t_n)\in M\times  M\times \R_+^\times$ be a sequence and $(gL,x,0)\in \cG\times \{0\}$. The following are equivalent
\begin{enumerate}
\item The sequence $(y_n,x_n,t_n)$ converges to $(gL,x,0)$ in $\mathbb{G}$
\item For any (or for some) graded basis $(V,\natural)$ and for any (or for some) $v\in V$, such that $\natural_{x,0}(v)\in gL$, there exists $v_n\in V$ such that $v_n\to v$, $x_n\to x$, $t_n\to 0$, $\ker(\natural_{x_n,t_n})\to L$, and $y_n=\exp(\natural_{t_n}(v_n))\cdot x_n$.
\end{enumerate}
\item Let $(g_nL_n,x_n,0)\in \cG\times \{0\}$ be a sequence, $(gL,x,0)\in \cG\times \{0\}$. The following are equivalent \begin{enumerate}
\item The sequence $(g_nL_n,x_n,0)$ converges to $(gL,x,0)$ in $\mathbb{G}$
\item For any (or for some) graded basis $(V,\natural)$ and for any (or for some) $v\in V$, such that $\natural_{x,0}(v)\in gL$, there exists $v_n\in V$ such that $v_n\to v$, $x_n\to x$, $\natural_{x_n,0}^{-1}(L_n)\to \natural_{x,0}^{-1}(L)$, and $\natural_{x_n,0}(v_n)\in g_nL_n$.
\end{enumerate}
\end{enumerate}
\end{prop}
For the reader familiar with Connes's tangent groupoid, the vector $v_n$ in Proposition \ref{prop:Conv}.1 is the equivalent of $t_n^{-1}(y_n-x_n)$ in the case of the tangent groupoid. In fact one way to think of the topology of $\mathbb{G}$ is that it captures the observation that if $(V,\natural)$ is a graded basis, $y$ and $x$ are close to each other and $t$ is small, then the space of solutions to the equation $y=\exp(\natural_t(v))\cdot x$ 'converges' to a coset of a subgroup of $\gr(\cF)_x$.

Let $(V,\natural)$ be a graded basis. The map $\cQ_{V}$ is continuous by definition, but in general it is not a quotient map. The issue here is the failure in general of the exponential map of vector fields to be a submersion away from $0$. This issue is easily fixed as the next proposition shows.
\begin{prop}\label{rem:cQ_V_quotient}Let $(V,\natural)$ be a graded basis. If $\mathbb{U}\subseteq \dom(\cQ_V)$ is an open subset, such that \footnote{For the reader familiar with \cite{AS1}, the set $\mathbb{U}$ is essentially what Androulidakis and Skandalis call a bisubmersion in our setting.} \begin{enumerate}
\item $V\times \cG^0\times \{0\}\subseteq \mathbb{U}$ 
\item The map  $\cQ_{V|\mathbb{U}\cap (M\times M\times \R_+^\times)}:\mathbb{U}\cap (M\times M\times \R_+^\times) \to M\times M\times \Rpt$ is a submersion.
\end{enumerate} 
Then the map $\cQ_{V|\mathbb{U}}:\mathbb{U}\to \mathbb{G}$ is open. Furthermore, an open subset $\mathbb{U}$ satisfying the above always exist.
\end{prop}
\begin{ex}Consider Example \ref{ex:mainex}. Let $\mu,\lambda_0,\cdots,\lambda_{N-1}\in \R$. Consider the vector field $X=\mu\partial_x+\sum_{i=0}^{N-1}\lambda_ix^i\partial_y$. A straightforward computation shows that\begin{equation}\label{eqn:ex_flow}
 \exp(X)\cdot (x,y)=\left(x+\mu,y+\sum_{i=0}^{N-1}\lambda_i\frac{(x+\mu)^{i+1}-x^{i+1}}{(i+1)\mu}\right).
\end{equation}By Proposition \ref{prop:Conv}, a sequence $((x_n',y_n'),(x_n,y_n),t_n)\in M\times M\times \R_+^\times$ converges to $(gL,(0,y),0)\in \cG\times \{0\}$ in $\mathbb{G}$ if and only if
 \begin{itemize}
 \item $(x_n,y_n)\to (0,y)$ and $t_n\to 0$
 \item The sequence $((x_n,y_n),t_n)\in M\times \R_+^\times$ converges in $\mathbb{G}^0$ to $(L,(0,y),0)$. So either $\frac{x_n}{t_n}\to \lambda$ for some $\lambda$ in which case $L$ is the subspace given by $\lambda$ or $\frac{x_n}{t_n}\to \infty$, in which case $L=\mathrm{span}([x^i\partial_y]:0\leq i\leq N-2)$.
\item There exists a sequence $\mu_n,\lambda_0^{(n)},\cdots,\lambda_{N-1}^{(n)}$  such that  \begin{align}\label{eqn:ex_comp}
 x_n'=x_n+t_n\mu_n,\quad y_n'=y_n+\sum_{i=0}^{N-1}t_n^{N}\lambda_i^{(n)}\frac{(\frac{x_n}{t_n}+\mu_n)^{i+1}-(\frac{x_n}{t_n})^{i+1}}{(i+1)\mu_n}
\end{align}
 and the sequences $\mu_n,\lambda_{N-1}^{(n)},\cdots,\lambda_0^{(n)}$ converge to $\mu,\lambda_{N-1},\cdots,\lambda_0$ and $$\mu [\partial_x]+\sum_{i=0}^{N-1}\lambda_i[x^i\partial_y]\in gL.$$
\end{itemize}

If $\frac{x_n}{t_n}\to \lambda$, then \eqref{eqn:ex_comp} implies that 
\begin{equation}\label{eqn:ex_1_1}
 \frac{x_n'-x_n}{t_n}\to \mu,\quad \frac{y_n'-y_n}{t_n^{N}}\to \sum_{i=0}^{N-1}\lambda_i\frac{(\lambda+\mu)^{i+1}-\lambda^{i+1}}{(i+1)\mu}
\end{equation} 
If $\frac{x_n}{t_n}\to \infty$, then \eqref{eqn:ex_comp} implies that
 \begin{equation}\label{eqn:ex_1_2}
\frac{x_n'-x_n}{t_n}\to \mu,\quad \frac{y_n'-y_n}{t_nx_n^{{N-1}}}\to \lambda_{{N-1}}.
\end{equation}
One can show the following \begin{itemize}
\item If $L=\mathrm{span}\Big(\lambda[x^{i-1}\partial_y]-[x^{i}\partial_y]:1\leq i\leq N-1\Big)$, then the cosets $gL$ are in bijection with pairs of real numbers $$\left(\mu, \sum_{i=0}^{N-1}\lambda_i\frac{(\lambda-\mu)^{i+1}-\lambda^{i+1}}{(i+1)\mu}\right)$$ 
where $$\mu [\partial_x]+\sum_{i=0}^{N-1}\lambda_i[x^i\partial_y]\in gL$$ is any element. 
If $L=\mathrm{span}([x^i\partial_y]:0\leq i\leq N-2)$,  then the cosets $gL$ are in bijection with pairs of real number s$$(\mu,\lambda_{N-1})$$ where $$\mu [\partial_x]+\sum_{i=0}^{N-1}\lambda_i[x^i\partial_y]\in gL$$ is any element.
\item If \eqref{eqn:ex_1_1} or \eqref{eqn:ex_1_2} is satisfied depending on $L$ then one can find sequences $(\mu_n)_{n\in \N},(\lambda_k^{(n)})_{n\in \N},\cdots,(\lambda_0^{(n)})_{n\in \N}$ that converge to $\mu,\lambda_{N-1},\cdots,\lambda_0$ and $\mu [\partial_x]+\sum_{i=0}^{N-1}\lambda_i[x^i\partial_y]\in gL$ and \eqref{eqn:ex_comp} is satisfied.
\end{itemize}
To summarize, $x_n\to 0$, $y_n\to y$, $t_n\to 0$ and $\frac{x_n}{t_n}$ converging in $\R\cup\{\infty\}$ together with \eqref{eqn:ex_1_1} or \eqref{eqn:ex_1_2} depending on $L$ completely determines the convergence $((x_n',y_n'),(x_n,y_n),t_n)\to (gL,(0,y),0)$.
\end{ex}
\textbf{Groupoid structure.} The space $\mathbb{G}$ is a groupoid with space of objects $\mathbb{G}^0$. The source and range maps are given by\begin{align*}
&s:\mathbb{G}\to \mathbb{G}^0,\quad s(y,x,t)=(x,t),\quad s(gL,x,0)=(L,x,0)\\
&r:\mathbb{G}\to \mathbb{G}^0,\quad r(y,x,t)=(y,t),\quad r(gL,x,0)=(gLg^{-1},x,0).
\end{align*}
The identity is given by $$\mathrm{id}:\mathbb{G}^0\to \mathbb{G},\quad (x,t)\mapsto (x,x,t),\quad (L,x,0)\mapsto (L,x,0).$$
The inverse is given by $$\iota:\mathbb{G}\to \mathbb{G},\ \iota(y,x,t)=(x,y,t),\quad \iota(gL,x,0)=((gL)^{-1},x,0)=(Lg^{-1},x,0).$$
The multiplication $m:\mathbb{G}\times_{r,s} \mathbb{G}\to \mathbb{G}$ is given by \begin{align*}
 (z,y,t)\cdot (y,x,t)=(z,x,t), (hgLg^{-1},x,0)\cdot(gL,x,0)&=((hgLg^{-1})(gL),x,0)\\&=(hgL,x,0).
\end{align*}
%changed here for duke
 It is straightforward to check that the above defines a groupoid structure.
\begin{theorem}\label{thm:cont_structure_maps}
The structure maps $s,r,\mathrm{id},\iota,m$ are continuous, where for the map $m$, the domain $\mathbb{G}\times_{r,s} \mathbb{G}$ is equipped with the subspace topology from $\mathbb{G}\times\mathbb{G}$. Furthermore, the maps $s,r$ are open.
\end{theorem}

\section{Proofs of the results in Sections \ref{sec:tangent_cones} and \ref{sec:tangent_groupoid}}\label{sec:proofs}
\subsection{Preliminary theorems}\label{subsubsec1}
\begin{dfn}A \textit{graded Lie basis} is a graded basis $(\mathfrak{g},\natural)$ where $\mathfrak{g}=\oplus_{i=1}^N\mathfrak{g}^i$ is equipped with the structure of a Lie bracket such that \begin{align*}
[\mathfrak{g}^i,\mathfrak{g}^j]\subseteq \begin{cases} \mathfrak{g}^{i+j},&i+
j\leq N\\0,&i
+j> N\end{cases}
\end{align*}
and such that \begin{equation}\label{eqn:bracket_cond_graded_lie_basis}
 \natural([v,w])=[\natural(v),\natural(w)],\quad v\in \mathfrak{g}^i,w\in \mathfrak{g}^j,i+j\leq N.
\end{equation}
\end{dfn}
Graded Lie basis always exist. For example, one can start with a graded basis $(V,\natural)$, then define $\mathfrak{g}$ as the graded free nilpotent Lie algebra of depth $N$ generated by elements of $V$ with the same grading as that of $V$. One then extends $\natural$ from $V$ to $\mathfrak{g}$ by using \eqref{eqn:bracket_cond_graded_lie_basis}. 
 
 In \cite[Appendix A]{MohsenMaxHypo}, we proved the following theorem. We remark that Appendix A can be read immediately and does not require any background from the rest of the article \cite{MohsenMaxHypo}. 
 \begin{theorem}[{\cite[Section A.2]{MohsenMaxHypo}}]\label{thm:composition_bisub}
 
 There exists $$\phi:\dom(\phi)\subseteq  \mathfrak{g}\times \mathfrak{g}\times M\times \R_+\to \mathfrak{g}$$ defined on some open set $\dom(\phi)$ 
such that \begin{enumerate}
\item $\mathfrak{g}\times \mathfrak{g}\times M\times \{0\}$ and $\{0\}\times \{0\}\times M\times \R_+$ are subsets of $\dom(\phi)$.
\item If $(v,w,x,t)\in \dom(\phi)$, then $\exp(\natural_t(\phi(v,w,x,t))\cdot x$ and $\exp\left(\natural_t(v)\right)\cdot\Big(\exp(	\natural_t(w))\cdot x\Big)$ are well-defined and $$\exp(\natural_t(\phi(v,w,x,t))\cdot x=\exp\left(\natural_t(v)\right)\cdot\Big(\exp(	\natural_t(w))\cdot x\Big).$$
\item For any  $(v,w,x,0)\in \mathfrak{g}\times \mathfrak{g}\times M\times \{0\}$, $\phi(v,w,x,0)=v\cdot w$ where $v\cdot w$ means the product of $v$ and $w$ in $\mathfrak{g}$ using the BCH formula.
\end{enumerate}
 \end{theorem}
 Let us explain the motivation behind Theorem \ref{thm:composition_bisub}. If $\natural:\mathfrak{g}\to \cX(M)$ is a Lie algebra homomorphism, then the space $\mathfrak{g}$, seen as a Lie group, acts (on the left by our convention on Lie algebras being right-invariant vector fields) on $M$.
  In particular, $$\exp(\natural_t(v))\cdot(\exp(\natural_t(w))\cdot x)=\exp(\natural_t(v\cdot w))\cdot x.$$
  Hence, one can take $\phi(v,w,x,t)=v\cdot w$. Therefore, in the general case, the map $\phi$ can be thought of as a correction to the absence of an action of $\mathfrak{g}$ on $M$.
 \begin{theorem}[{\cite[Section A.3]{MohsenMaxHypo}}]\label{thm:proof_different_bases} Let $(V,\natural)$ and $(V',\natural')$  be graded basis. Then there exists a smooth map $$\phi:\dom(\phi)\subseteq V\times M\times \R_+\to V'$$ defined on some open subset $\dom(\phi)$ such that 
\begin{enumerate}
\item $V\times M\times \{0\}$ and $\{0\}\times M\times \R_+$ are subsets of $\dom(\phi)$.
\item If $(v,x,t)\in \dom(\phi)$, then $\exp(\natural_t(v))\cdot x$ and $\exp(\natural_t'(\phi(v,x,t)))\cdot x$ are well-defined and $$\exp(\natural_t(v))\cdot x=\exp(\natural_t'(\phi(v,x,t)))\cdot x.$$
\item For any $(v,x,0)\in V\times M\times \{0\}$, then $\natural_{x,0}(v)=\natural'_{x,0}(\phi(v,x,0))$.
\end{enumerate}
\end{theorem}

 We need the periodic bounding lemma \cite{period2}, see \cite[Proposition 1.1]{Debord2013} for a proof of the version we use.
 \begin{theorem}[Periodic bounding lemma]\label{thm:periodicbounding} Let $(V,\natural)$ be a graded basis, $\norm{\cdot}$ a norm on $V$, $K\subseteq M$ a compact subset. Then there exists $\epsilon>0$, such that for any $x\in K$, $v\in V$ such that $\exp(\natural(v))\cdot x=x$, either $\natural(v)(x)=0$ or $\norm{v}\geq \epsilon$.
 \end{theorem}
The following simple proposition about Grassmannian manifolds will be used repeatedly.
\begin{prop}\label{prop:conv_Grass} Let $V$ be a vector space, $0\leq d\leq \dim(V)$, $\Grass_d(V)$ the Grassmannian manifold of subspaces of dimension $d$, $(L_n)_{n\in \N}\in \Grass_d(V)$ a sequence, $L\in \Grass_d(V)$. The following are equivalent \begin{enumerate}
\item $L_n$ converges to $L$ in $\Grass_d(V)$
\item For any sequence $v_n\in L_n$ if $v_n$ converges in $V$, then the limit is in $L$.
\end{enumerate}Furthermore, if $L_n\to L$, then for any $v\in L$, there exists $v_n\in L_n$ such that $v_n\to v$.
\end{prop}
\subsection{Proof of Proposition \ref{thm:top_well_defined}}\label{sec:proof_prop_well_def_top}
For Part 1, it is enough to show that if $x_n\to x$ and $t_n\to 0$ and $\ker(\natural_{x_n,t_n})\to L$, where $L\in \Grass(V)$, then $\ker(\natural_{x,0})\subseteq L$. The map $\natural_{x,0}$ preserves the grading on $V$ and $\gr(\cF)_{x}$. Hence, it is enough to consider $v\in \ker(\natural_{x,0})\cap V^k$ for some $k$. Hence, $\natural(v)\in \cF^{k-1}+I_x\cF^k$. By Definition \ref{rem:simpler_basis}, we get that $$\natural(v)=\sum_{i\in I} f_i \natural(v_i)+\sum_{j\in J}g_j\natural(w_j),$$ where $I,J$ are some finite index sets, $v_i\in V^{\omega(i)}$ with $\omega(i)<k$, $w_j\in V^k$, $f_i,g_j\in C^\infty(M,\R)$ and $g_j(x)=0$. Since $$v-\sum_{i\in I}f_i(x_n)t_n^{k-\omega(i)}v_i-\sum_{j\in J}g_j(x_n)w_j\in \ker(\natural_{x_n,t_n})$$ and $$v-\sum_{i\in I}f_i(x_n)t_n^{k-\omega(i)}v_i-\sum_{j\in J}g_j(x_n)w_j\xrightarrow{n\to +\infty} v,$$ it follows from Proposition \ref{prop:conv_Grass} that $v\in L$.

Part 2 follows immediately from Part 1. For Part 3, let $(V,\natural)$ and $(V',\natural')$ be graded basis, $(x_n,t_n)\in M\times  \R_+^\times$ such that $(x_n,t_n)\to (x,0)$ and $\ker(\natural_{x_n,t_n})\to \natural_{x,0}^{-1}(L)$. We need to show that $\ker(\natural'_{x_n,t_n})\to \natural_{x,0}^{\prime-1}(L)$. Since Grassmannian manifolds are compact, by passing to a subsequence we can suppose that $\ker(\natural'_{x_n,t_n})$ converges, and by Part $1$ that we just proved, the limit has to be of the form $\natural_{x,0}^{\prime-1}(L')$ for some $L'\in \Grass(\gr(\cF)_x)$. We need to show that $L=L'$. Let $\phi$ be as in Theorem \ref{thm:proof_different_bases}, $$\psi:V\times M\times \R_+\to V',\quad \psi(v,y,t)=\frac{d}{ds}\Big\rvert_{s=0}\phi(sv,y,t).$$ Since $\phi$ is smooth, $\psi$ is smooth and linear in $v$. Let $v\in \ker(\natural_{y,t})$ for $t\in \R_+^\times$. By Theorem \ref{thm:proof_different_bases}.2, $$\exp(\natural_{y,t}'(\phi(sv,y,t)))\cdot y=\exp(\natural_{y,t}(sv))\cdot y=y,\quad \text{for }s\text{ small enough}.$$ Hence, by Theorem \ref{thm:periodicbounding}, $\psi(v,y,t)\in \ker(\natural_{y,t}')$. Now let $l\in \natural_{x,0}^{-1}(L)$. By Proposition \ref{prop:conv_Grass}, there exists $l_n\in\ker(\natural_{x_n,t_n}) $ such that $l_n\to l$. Hence, $\psi(l_n,x_n,t_n)\to \psi(l,x,0)$. We just proved that $\psi(l_n,x_n,t_n)\in \ker(\natural'_{x_n,t_n})$. Hence, by Proposition \ref{prop:conv_Grass}, $\psi(l,x,0)\in \natural_{x,0}^{\prime -1}(L')$. By Theorem \ref{thm:proof_different_bases}, it follows that $\natural_{x,0}(l)=\natural'_{x,0}(\psi(l,x,0))$. Hence, $\natural_{x,0}(l)\in L'$. Since $\natural_{x,0}$ is surjective, we proved that $L\subseteq L'$. Since both have the same dimension, we get that $L=L'$.
 \subsection{Proof of Theorem \ref{thm:prop of blowup}}\label{sec:proof_of_prop_of_blowup}
Let $(\mathfrak{g},\natural)$ be a graded basis. The map $\natural_{x,0}:\mathfrak{g}\to \gr(\cF)_x$ is a Lie algebra homomorphism because of \eqref{eqn:bracket_cond_graded_lie_basis}. Hence, it is also a Lie group homomorphism. Let $x_n\in M$, $t_n\in \R_+^\times$, such that $x_n\to x$, $t_n\to 0$, $\ker(\natural_{x_n,t_n})\to \natural_{x,0}^{-1}(L)$, where $L\in \cG^0_x$.
\begin{lem}\label{lem:qjsidfj} Let $g_n\in \mathfrak{g}$ converge to $g\in \mathfrak{g}$, $y_n=\exp(\natural_{t_n}(g_n))\cdot x_n$. Then $y_n\to x$ and $$\ker(\natural_{y_n,t_n})\to g\natural_{x,0}^{-1}(L)g^{-1}=\natural_{x,0}^{-1}\left(\natural_{x,0}(g)L\natural_{x,0}(g)^{-1}\right).$$
\end{lem}
\begin{proof}The last equality follows from the fact that $\natural_{x,0}$ is a Lie group homomorphism. Let $\phi$ as in Theorem \ref{thm:composition_bisub}. Consider the map $\psi:\dom(\psi)\subseteq \mathfrak{g}\times \mathfrak{g}\times M\times \R_+\to \mathfrak{g}$ defined by%Duke
  $$ \psi(g,v,y,t)=\frac{d}{ds}\Big\rvert_{s=0} \phi\Big(\phi(g,sv,y,t),-g,\exp(\natural_t(g))\cdot y,t\Big).$$The domain of $\psi$ is the set of points for which the right-hand side is well-defined. Let $v\in \ker(\natural_{y,t})$. In the next equation 
  $z=\exp(\natural_t(g))\cdot y$, then  
  \begin{align*}
&\exp\Big(\natural_{t}\Big(\phi(\phi(g,sv,y,t),-g,z,t)\Big)\Big)\cdot z\\
&=\exp(\natural_{t}(\phi(g,sv,y,t)))\cdot\Big(\exp(\natural_{t}(-g))\cdot z\Big)\\&=\exp(\natural_{t}(\phi(g,sv,y,t)))\cdot y\\&=\exp(\natural_{t}(g))\cdot\Big(\exp(\natural_{t}(sv))\cdot y\Big)\\&=\exp(\natural_{t}(g))\cdot y=z.
\end{align*}
It follows from Theorem \ref{thm:periodicbounding}, that $$\psi(g,v,y,t)\in \ker(\natural_{\exp(\natural_t(g))\cdot y,t}),\quad t\neq 0.$$

Now we can prove the lemma. It is clear that $y_n\to x$. By compactness of Grassmannian manifolds, we can by passing to a subsequence, suppose that $\ker(\natural_{y_n,t_n})$ converges, which by Proposition \ref{thm:top_well_defined}.1 must be of the form $\natural_{x,0}^{-1}(L')$ for some $L'\in \cG^0_x$. Let $v\in \natural_{x,0}^{-1}(L)$. By Proposition \ref{prop:conv_Grass}, there exists $v_n \in \ker(\natural_{x_n,t_n})$ such that $v_n\to v$. Notice that as $n\to +\infty$, $(g_n,v_n,x_n,t_n)\to (g,v,x,0)\in \dom(\psi)$. So without loss of generality, we suppose that $(g_n,v_n,x_n,t_n)\in \dom(\psi)$. Since $\psi$ is continuous (even smooth), we get that $\psi(g_n,v_n,x_n,t_n)\to \psi(g,v,x,0)$.  By Theorem \ref{thm:composition_bisub}.3, $\psi(g,v,x,0)=gvg^{-1}$. Since $\psi(g_n,v_n,x_n,t_n)\in \ker(\natural_{y_n,t_n})$, it follows that $gvg^{-1}\in \natural_{x,0}^{-1}(L')$. We have thus proved that $g\natural_{x,0}^{-1}(L)g^{-1}\subseteq \natural_{x,0}^{-1}(L')$. Since both are of the same dimension (equal to $\dim(\mathfrak{g})-\dim(M)$), we get that $g\natural_{x,0}^{-1}(L)g^{-1}=\natural_{x,0}^{-1}(L')$.
\end{proof}
Lemma \ref{lem:qjsidfj} immediately implies that $\natural_{x,0}(g)L\natural_{x,0}(g)^{-1}\in \cG^0_x$, which is Theorem \ref{thm:prop of blowup}.2. For Theorem \ref{thm:prop of blowup}.1, let $\hat{g}\in L$. By surjectivity of $\natural_{x,0}$, let $g\in \natural_{x,0}^{-1}(L)$ such that $\natural_{x,0}(g)=\hat{g}$. By Proposition \ref{prop:conv_Grass}, there exists $g_n\in \ker(\natural_{x_n,t_n})$ such that $g_n\to g$. By Lemma \ref{lem:qjsidfj}, $\ker(\natural_{x_n,t_n})\to \natural_{x,0}^{-1}\left(\hat{g}L\hat{g}^{-1}\right)$. Hence, $\hat{g}L\hat{g}^{-1}=L$, which implies that $L$ is a Lie subalgebra and a Lie subgroup. 
\subsection{Proof of Theorem \ref{thm:top_on_bG}.1}
Let $(V,\natural)$ and $(V',\natural')$ be graded basis. Let $U\subseteq \mathbb{G}$ such that $U\cap (M\times M\times \R_+^\times)$ and $\cQ^{-1}_{V'}(U)$ are open. We need to show that $\cQ_{V}^{-1}(U)$ is open. The set $\cQ_{V}^{-1}(U)\cap (V\times M\times \R_+^\times)$ is clearly open. Let $\phi$ be as in Theorem \ref{thm:proof_different_bases}. We define the set $\dom(\hat{\phi})=\pi^{-1}(\dom(\phi))$, where $\pi:V\times \mathbb{G}^0\to V\times M\times \R_+$ is the product of the identity map and the natural projection $\mathbb{G}^0\to M\times \R_+$. We also define the map \begin{equation}\label{eqn:change_of_basis}
\begin{aligned}
\hat{\phi}:\dom(\hat{\phi})\subseteq V\times \mathbb{G}^0&\to V'\times \mathbb{G}^0\\
(v,y,t)&\mapsto (\phi(v,y,t),y,t)\quad t\neq 0\\
(v,L,y,0)&\mapsto (\phi(v,y,0),L,y,0)
\end{aligned}
\end{equation}
The map $\hat{\phi}$ is continuous because $\phi$ is continuous. The following diagram commutes $$\begin{tikzcd}
\dom(\hat{\phi})\arrow[dr,"\cQ_{V}"']\arrow[r,"\hat{\phi}"]& V'\times \mathbb{G}^0\arrow[d,"\cQ_{V'}"]\\&\mathbb{G}
\end{tikzcd}.$$Hence, $\cQ_{V}^{-1}(U)\cap \dom(\hat{\phi})$ is open. Since $V\times \cG^0\times \{0\}\subseteq \dom(\hat{\phi})$, we deduce that $\cQ_{V}^{-1}(U)$ is open.
\subsection{Proof of Theorem \ref{thm:cont_structure_maps}.}
We fix a graded Lie basis $(\mathfrak{g},\natural)$. Theorem \ref{thm:top_on_bG}.1 implies that the continuity of a map $f:\mathbb{G}\to X$ where $X$ is any topological space, is equivalent to the continuity of $f_{|M\times M\times \R_+^\times}$ and $f\circ \cQ_{\mathfrak{g}}$.

 \textbf{Continuity of $s$.}
  The maps $s_{|M\times M\times \R_+^\times}:M\times M\times \R_+^\times\to \mathbb{G}^0$ and $s\circ \cQ_{\mathfrak{g}}:\mathfrak{g}\times \mathbb{G}^0\to \mathbb{G}^0$ are continuous. Continuity of the first map is easy to check. The second is the projection map $\mathfrak{g}\times \mathbb{G}^0\to \mathbb{G}^0$. 

  \textbf{Continuity of $r$.} We proceed as we did with $s$. The continuity of the map $r_{|M\times M\times \R_+^\times}:M\times M\times \R_+^\times\to \mathbb{G}^0$ is easy to check. For the map $r\circ \cQ_\mathfrak{g}$, we remark that $r\circ \cQ_\mathfrak{g}$ is given by \begin{align*}
r\circ \cQ_\mathfrak{g}:\mathfrak{g}\times \mathbb{G}^0&\to  \mathbb{G}^0\\
(g,x,t)&\mapsto (\exp(\natural_t(g))\cdot x,t)\\
(g,L,x,0)&\mapsto (\natural_{x,0}(g)L\natural_{x,0}(g)^{-1},x,0)
\end{align*}
Lemma \ref{lem:qjsidfj} says that if $(g_n,x_n,t_n)\in \mathfrak{g}\times M\times \R^\times_+$ converges to $(g,L,x,0)\in \mathfrak{g}\times \mathbb{G}^0$, then $r\circ \cQ_\mathfrak{g}(g_n,x_n,t_n)\to r\circ \cQ_\mathfrak{g}(g,L,x,0)$. It is immediate to check that if $(g_n,L_n,x_n,0)\in \mathfrak{g}\times \mathbb{G}^0$ converges to $(g,L,x,0)\in \mathbb{G}^0$, then $r\circ \cQ_\mathfrak{g}(g_n,L_n,x_n,0)\to r\circ \cQ_\mathfrak{g}(g,L,x,0)$. Hence, $r\circ \cQ_\mathfrak{g}$ is continuous.

\textbf{Continuity of $\mathrm{Id}$.} Let $\phi:\mathbb{G}^0\to \mathfrak{g}\times \mathbb{G}^0$ be the map $\phi(a)=(0,a)$ for $a\in \mathbb{G}^0$. Then the diagram $$\begin{tikzcd}\mathbb{G}^0\arrow[r,"\phi"]\arrow[dr,"\mathrm{id}"']&\mathfrak{g}\times\mathbb{G}^0\arrow[d,"\cQ_{\mathfrak{g}}"]\\&\mathbb{G}\end{tikzcd}$$ commutes. It follows that $\mathrm{id}$ is continuous. 

\textbf{Continuity of $\iota$.} Notice that $\iota(M\times M\times \R_+^\times)= M\times M\times \R_+^\times$, and $\iota_{|M\times M\times \R_+^\times}:M\times M\times \R_+^\times\to M\times M\times \R_+^\times$ is continuous. Hence, it suffices to prove that $\iota\circ \cQ_{\mathfrak{g}}$ is continuous. We define the map \begin{align*}
\phi:\mathfrak{g}\times \mathbb{G}^0&\to  \mathfrak{g}\times \mathbb{G}^0\\
(g,x,t)&\mapsto (-g,\exp(\natural_t(g))\cdot x,t)\\
(g,L,x,0)&\mapsto (-g,\natural_{x,0}(g)L\natural_{x,0}(g)^{-1},x,0)
\end{align*}
Continuity of the map $\phi$ follows from continuity of $r\circ \cQ_{\mathfrak{g}}$. One can easily check that the diagram $$\begin{tikzcd}
\mathfrak{g}\times \mathbb{G}^0\arrow[d,"\cQ_{\mathfrak{g}}"]\arrow[r,"\phi"]&\mathfrak{g}\times \mathbb{G}^0\arrow[d,"\cQ_{\mathfrak{g}}"]\\\mathbb{G}\arrow[r,"\iota"]&\mathbb{G}
\end{tikzcd}$$commutes. It follows that $\iota$ is continuous.

\textbf{Continuity of $m$.} The restriction of $m$ to the $M\times M\times \R_+^\times$ part of $\mathbb{G}$ is continuous. Consider the product
 $$(\mathfrak{g}\times \mathbb{G}^0)\times_{s\circ\cQ_{\mathfrak{g}},r\circ \cQ_{\mathfrak{g}}}\times (\mathfrak{g}\times \mathbb{G}^0).$$ Since the map $s\circ \cQ_{\mathfrak{g}}$ is the projection onto $\mathbb{G}^0$, 
 it follows that we can identify the product with $\mathfrak{g}\times \mathfrak{g}\times \mathbb{G}^0$. 
 It suffices to show that the map $$m\circ (\cQ_{\mathfrak{g}}\times \cQ_{\mathfrak{g}}):\mathfrak{g}\times \mathfrak{g}\times \mathbb{G}^0\to \mathbb{G}$$ is continuous at a point of the form $(g,h,x,L,0)$. Let $\phi$ be as in Theorem \ref{thm:composition_bisub}. We define $\dom(\hat{\phi})=\pi^{-1}(\dom(\phi))$, where $\pi:\mathfrak{g}\times \mathfrak{g}\times \mathbb{G}^0\to \mathfrak{g}\times \mathfrak{g}\times M\times \R_+$ is the product of the identity map on $\mathfrak{g}\times\mathfrak{g}$ and the natural projection $\mathbb{G}^0\to M\times \R_+$. We also define the map \begin{equation}\label{eqn:m_product}\begin{aligned}
\hat{\phi}:\dom(\hat{\phi})&\to \mathfrak{g}\times \mathfrak{g}\times \mathbb{G}^0\\
(g,h,y,t)&\mapsto (\phi(g,h,y,t),y,t),\quad t\neq 0\\
(g,h,L,y,0)&\mapsto (\phi(g,h,y,0),L, y,0 )=(g\cdot h, L,y,0 )
\end{aligned}
\end{equation}
The map $\hat{\phi}$ is continuous because $\phi$ is continuous. The diagram $$\begin{tikzcd}\dom(\hat{\phi})\arrow[r,"\hat{\phi}"]\arrow[dr,"m\circ (\cQ_{\mathfrak{g}}\times \cQ_{\mathfrak{g}})"']&\mathfrak{g}\times \mathbb{G}^0\arrow[d,"\cQ_{\mathfrak{g}}"]\\& \mathbb{G}
\end{tikzcd}$$commutes. Continuity of $m:\mathbb{G}\times_{r,s} \mathbb{G}\to \mathbb{G}$ follows.

\textbf{The maps $s$ and $r$ are open.} Let $U\subseteq \mathbb{G}$ be an open subset. Then $$s(U)=s(U\cap (M\times M\times \R_+^\times))\cup s\circ \cQ_{\mathfrak{g}}(\cQ_{\mathfrak{g}}^{-1}(U)).$$ The map $s\circ \cQ_{\mathfrak{g}}$ is a projection map, 
hence open. It follows that $s$ is open. Since $\iota\circ \iota$ is the identity, it follows that $\iota$ is a homeomorphism. Hence, $r=s\circ \iota$ is also open.
\subsection{Proof of Proposition \ref{rem:cQ_V_quotient}}
First for the existence of $\mathbb{U}$ satisfying $1$ and $2$, let $(V,\natural)$ be a graded basis. The map $$\psi:V\times M\to M\times M,\quad \psi(v,x)=(\exp(\natural(v))\cdot x,x)$$ is well-defined in a neighbourhood of $\{0\}\times M$. Let $U\subseteq \dom(\psi)$ be an open neighbourhood of $\{0\}\times M$ such that $\psi$ is a submersion at every point in $U$. We define $\mathbb{U}\subseteq V\times \mathbb{G}^0\to \mathbb{G}$ by
\begin{align*}
\mathbb{U}=V\times \cG^0\times \{0\}\sqcup \{(v,x,t)\in V\times M\times \R_+^\times:(\sum_{i=1}^Nt^iv_i,x)\in U\},
\end{align*}
where $v=\sum_{i=1}^Nv_i$ is the decomposition of $v$ in $V=\oplus_{i=1}^NV^i$. The set $\mathbb{U}$ is easily seen to be an open subset of $V\times \mathbb{G}^0$ which satisfies $1$ and $2$.

Now let $(V,\natural)$ be a graded basis, $\mathbb{U}\subseteq V\times \mathbb{G}^0$ an open subset satisfying $1$ and $2$.
\begin{lem}\label{lem:sqjdfjlq} 
  If $v,v'\in V$, $(L,x,0)\in \mathbb{G}^0$ such that we have $\cQ_V(v,L,x,0)=\cQ_V(v',L,x,0)$. Then there exists a continuous partially defined function $\hat{\phi}:V\times \mathbb{G}^0\to V\times \mathbb{G}^0$ such that $\hat{\phi}(v',L,x,0)=(v,L,x,0)$ and $\cQ_{V}\circ \hat{\phi}=\cQ_{V}$.
\end{lem}
\begin{proof}
First we prove the lemma for a graded Lie basis $(\mathfrak{g},\natural)$. Let $v,v'\in \mathfrak{g}$ such that \begin{equation}
 \cQ_\mathfrak{g}(v,L,x,0)=\cQ_\mathfrak{g}(v',L,x,0).
\end{equation}This means that $v^{\prime-1}v\in \natural_{x,0}^{-1}(L)$. On the space $\Grass(\mathfrak{g})\times M\times \R_+$, there is the tautological vector bundle whose fiber at $(A,y,t)$ is $A$. Let $f:\Grass(\mathfrak{g})\times M\times \R_+\to \mathfrak{g}$ be any section of this vector bundle locally defined in a neighbourhood of $(\natural_{x,0}^{-1}(L),x,0)$ with $f(\natural_{x,0}^{-1}(L),x,0)=v^{\prime-1}v$. By composing $f$ with $\Grass(\natural)$ in \eqref{eqn:Grass_natural}, we get a partially defined continuous map $f:\mathbb{G}\to \mathfrak{g}$ such that $f(L,x,0)=v^{\prime-1}v$ and if $t\neq 0$, then $f(y,t)\in \ker(\natural_{y,t})$. Hence, $\exp(\natural_{t}(f(y,t)))\cdot y=y$. Let $\phi$ be as in Theorem \ref{thm:composition_bisub}. Then we have a partially defined map in a neighbourhood of $(v,L,x,0)$  \begin{align*}
\hat{\phi}:\mathfrak{g}\times \mathbb{G}^0&\to \mathfrak{g}\times \mathbb{G}^0\\
(w,y,t)&\mapsto (\phi(w,f(y,t),y,t),y,t)\\(w,K,y,0)&\mapsto(\phi(w,f(y,K,0),y,0),K,y,0)\quad w\in \mathfrak{g}, y\in M,K\in \cG^0_y
\end{align*}
The map $\hat{\phi}$ is continuous because $\phi$ and $f$ are continuous. The map $\hat{\phi}$ sends $(v',L,x,0)$ to $(v,L,x,0)$ and $\cQ_{\mathfrak{g}}\circ \hat{\phi}=\cQ_{\mathfrak{g}}$.

Now let $(V,\natural)$ be any graded basis with $v,v'\in V$ such that $\cQ_V(v,L,x,0)=\cQ_V(v',L,x,0)$. Let $(\mathfrak{g},\natural)$ be the graded Lie basis generated by $(V,\natural)$ as we did after \eqref{eqn:bracket_cond_graded_lie_basis}. 
So $V\subseteq \mathfrak{g}$ and $\natural:\mathfrak{g}\to \cX(M)$ extends $\natural:V\to \cX(M)$. Hence, $\cQ_{\mathfrak{g}}(v,L,x,0)=\cQ_V(v,L,x,0)$. Let $\psi_1:\mathfrak{g}\times M\times \R_+\to V$ be the map obtained from Theorem \ref{thm:comparison}. 
Then the construction of $\psi_1$ in \cite[Appendix A]{MohsenMaxHypo} has the property that $\psi_1(w,x,0)=(w,x,0)$ for all $w\in V$.  
Now let $\psi_2:V\times M\times \R_+\to \mathfrak{g}$ be the map also obtained from Theorem \ref{thm:comparison}. Then take $v''=\psi_2(v',x,0)\in \mathfrak{g}$. Then $\cQ_\mathfrak{g}(v'',L,x,0)=\cQ_V(v',L,x,0)$. 
So $\cQ_\mathfrak{g}(v'',L,x,0)=\cQ_\mathfrak{g}(v,L,x,0)$. Hence, there exists $\hat{\phi}:\mathfrak{g}\times \mathbb{G}^0\to \mathfrak{g}\times \mathbb{G}^0$ that sends $(v'',L,x,0)$ to $(v,L,x,0)$ and preserves $\cQ_\mathfrak{g}$. Let $\hat{\psi}_1:\mathfrak{g}\times \mathbb{G}^0\to V\times\mathbb{G}^0$ and $\hat{\psi}_2:V\times\mathbb{G}^0\to \mathfrak{g}\times\mathbb{G}^0$ as in \eqref{eqn:change_of_basis} associated to $\psi_1$ and $\psi_2$. Then the required map is $\hat{\psi}_1\circ \hat{\phi}\circ \hat{\psi}_2$.
\end{proof}
 Let $W\subseteq \mathbb{U}$ be an open subset. By our hypothesis on $\mathbb{U}$, $\cQ_{V}(W)\cap M\times M\times \R_+^\times$ is an open subset of $M\times M\times \R_+^\times$. By Theorem \ref{thm:top_on_bG}.1, to show that $\cQ_{V}(W)$ is open it suffices to show that $\cQ_{V}^{-1}(\cQ_V(W))$ is open. Hence, it is enough to show that for any $(v,L,x,0)\in W$, and any point $(v',L,x,0)\in  V\times \mathbb{G}^0$ such that $\cQ_V(v',L,x,0)=\cQ_V(v,L,x,0)$, the set $\cQ_{V}^{-1}(\cQ_V(W))$ contains a neighbourhood of $(v',L,x,0)$. By Lemma \ref{lem:sqjdfjlq}, $\cQ_{V}^{-1}(\cQ_V(W))$ contains the neighbourhood $\hat{\phi}^{-1}(W)$ of $(v',L,x,0)$.
\subsection{Proof of Theorem \ref{thm:top_on_bG}.2}\label{sec:proof_Hasud}
We first show that $\mathbb{G}$ is Hausdorff. By Theorem \ref{thm:cont_structure_maps}, the map $s:\mathbb{G}\to \mathbb{G}^0$ is continuous. The space $\mathbb{G}^0$ is Hausdorff by Proposition \ref{thm:top_well_defined}.
Hence, it remains to show that one can separate two different points of the form $(gL,x,0)$ and $(hL,x,0)$ with $g,h\in \gr(\cF)_x$. Let $(\mathfrak{g},\natural)$ be a graded Lie basis. 
We pick $\tilde{g}, \tilde{h}\in \mathfrak{g}$ such that $\natural_{x,0}(\tilde{g})=g$ and $\natural_{x,0}(\tilde{h})=h$. The points $(\tilde{g},L,x,0)$ and $(\tilde{h},L,x,0)$ can be separated in $\mathfrak{g}\times \mathbb{G}^0$ by open sets 
$U_1$ and $U_2$. We can further suppose that $U_1,U_2\subseteq \mathbb{U}$, where $\mathbb{U}$ is given by Proposition \ref{rem:cQ_V_quotient}. Hence, $\cQ_{\mathfrak{g}}(U_1)$ and $\cQ_{\mathfrak{g}}(U_2)$ are open. 
We claim that for $U_1$ and $U_2$ small enough $\cQ_{\mathfrak{g}}(U_1)\cap\cQ_{\mathfrak{g}}(U_2)=\emptyset$. Suppose that this is not possible. It is straightforward to see that for $U_1$ and $U_2$ small enough 
$\cQ_{\mathfrak{g}}(U_1)\cap\cQ_{\mathfrak{g}}(U_2)\cap \cG\times \{0\}=\emptyset$. Hence, it follows that there exists two sequences $(\tilde{g}_n,x_n,t_n)$ and $(\tilde{h}_n,x_n,t_n)$ with $t_n>0$ such that 
$\cQ_{\mathfrak{g}}(\tilde{g}_n,x_n,t_n)=\cQ_{\mathfrak{g}}(\tilde{h}_n,x_n,t_n)$ and $(\tilde{g}_n,x_n,t_n)\to (\tilde{g},L,x,0)$ and $(\tilde{h}_n,x_n,t_n)\to (\tilde{h},L,x,0)$ in $\mathfrak{g}\times \mathbb{G}^0$. Hence, 
$$\exp(\natural_{t_n}(\tilde{g}_n))\cdot x_n=\exp(\natural_{t_n}(\tilde{h}_n))\cdot x_n$$ which implies $$\exp(\natural_{t_n}(-\tilde{h}_n))\cdot(\exp(\natural_{t_n}(\tilde{g}_n))\cdot x_n)= x_n$$ Let $\phi$ be as in Theorem \ref{thm:composition_bisub}. 
Then for $n$ big enough $(-\tilde{h}_n,\tilde{g}_n,x_n,t_n)\in \dom(\phi)$ because $(-\tilde{h}_n,\tilde{g}_n,x_n,t_n)\to (-\tilde{h},\tilde{g},x,0)\in \dom(\phi)$. Also, we have $\phi(-\tilde{h}_n,\tilde{g}_n,x_n,t_n)\to \phi(-\tilde{h},\tilde{g},x,0)=\tilde{h}^{-1}\tilde{g}$. Furthermore, $$\exp(\natural_{t_n}(\phi(-\tilde{h}_n,\tilde{g}_n,x_n,t_n)))\cdot x_n=\exp(\natural_{t_n}(-\tilde{h}_n))\cdot\left(\exp(\natural_{t_n}(\tilde{h}_n))\cdot x_n\right)=x_n.$$ By Theorem \ref{thm:periodicbounding}, we deduce that for $n$ big enough $\phi(-\tilde{h}_n,\tilde{g}_n,x_n,t_n)\in \ker(\natural_{x_n,t_n})$. Since $(x_n,t_n)\to (L,x,0)$ in $\mathbb{G}^0$, we deduce $\ker(\natural_{x_n,t_n})\to  \natural_{x,0}^{-1}(L)$. By Proposition \ref{prop:conv_Grass}, $\tilde{h}^{-1}\tilde{g}\in \natural_{x,0}^{-1}(L)$. This implies that $gL=hL$ and $(gL,x,0)$ and $(hL,x,0)$  are an identical point which is a contradiction.

Proposition \ref{rem:cQ_V_quotient} implies that $\cQ_{\mathfrak{g}}(\mathbb{U})$ is second countable. Since $\mathbb{G}=(M\times M\times \R_+^\times) \cup \cQ_{\mathfrak{g}}(\mathbb{U})$, it follows that $\mathbb{G}$ is second countable. Again Proposition \ref{rem:cQ_V_quotient} implies that $\mathbb{G}$ is locally compact. Hence, it is a $T_3$ topological space. By Urysohn's metrization theorem, $\mathbb{G}$ is metrizable.
\subsection{Proof of Proposition \ref{prop:conv_Grass}}
The proof of the two parts is similar. The implication $b\implies a$ (with for some in both cases) follows from the continuity of $\cQ_V$. The implication $a\implies b$ follows from Proposition \ref{rem:cQ_V_quotient} and Theorem \ref{thm:top_on_bG}.2.
\section{Continuous family groupoids}\label{sec:cont_fam_grp}
In this section, we prove that the groupoid $\mathbb{G}\rightrightarrows \mathbb{G}^0$ is naturally a continuous family groupoid. Continuous family groupoids were introduced by Paterson \cite{ContFamilyGroupoids}. We recall here the definition. We first need the following terminology \begin{dfn} Let $X$ be a topological space and $g:\R^n\times X\to \R^m$ a continuous function. Then we say that $g$ is $C^{\infty,0}$ if $g(\cdot,x)$ is smooth for any $x\in X$, and for any finite index set $I$, $ \frac{\partial}{\partial_It}g:\R^n\times X\to \R^m$ is continuous.

More generally let $Y$ be a topological space, $f:\R^n\times X\to \R^m\times Y$ a continuous function which is of the form $f(t,x)=(g(t,x),h(x))$ for some $g:\R^n\times X\to \R^m$ and $h:X\to Y$. We say that $f$ is $C^{\infty,0}$ if $g$ is $C^{\infty,0}$.
\end{dfn}
\begin{dfn} \label{dfn:cont_fam}

A continuous family groupoid with $s$-fibers of dimension $n\in \N$ is a groupoid $G\rightrightarrows G^0$ such that \begin{enumerate}
\item The set $G$ is equipped with a locally compact Hausdorff topology, and $G^0\subseteq G$ is equipped with the subspace topology.
\item The structure maps $s,r,\iota,m$ are continuous, where $G^{(2)}\subseteq G\times G$ is equipped with the subspace topology.
\item We have a family $\cA$ of pairs $(\Omega,\phi)$ called charts of $G$, where $\Omega\subseteq G$ is an open subset, and $\phi:\Omega\to \R^n\times G^{0}$ is an open embedding such that $G=\bigcup_{(\Omega,\phi)\in \cA}\Omega$ and the following diagram commutes 
\begin{equation}\label{eqn:Diagram_Cont_Fam}
 \begin{tikzcd}
\Omega\arrow[r,"\phi"]\arrow[d,"s"]&\R^n\times G^{0} \arrow[dl,"p"]\\G^0
\end{tikzcd}
\end{equation}
where $p$ is the projection onto the second coordinate.
\item If $(\Omega,\phi)$ and $(\Omega',\phi')$ are charts, then the partially defined map \begin{align}\label{eqn:map_psiqsndlfjk}
\psi:\R^n\times G\to \R^n \times G^0,\quad \psi(t,\gamma)=\phi^{\prime}(\phi^{-1}(t,r(\gamma))\cdot\gamma)
\end{align} is $C^{\infty,0}$. By restricting the map $\psi$ to $\R^n\times G^0$, we get that  $\phi'\circ \phi^{-1}$ is $C^{\infty,0}$ on its domain.
\item The family $\cA$ is maximal, i.e, if $\Omega\subseteq G$ is an open subset and $\phi:\Omega\to \R^n\times G^{0}$ is an open embedding such that \eqref{eqn:Diagram_Cont_Fam} commutes and for every $(\Omega',\phi')\in \cA$, the map $\phi'\circ \phi^{-1}$ is $C^{\infty,0}$, then $(\Omega,\phi)\in \cA$.
\end{enumerate}
\end{dfn}

 A way of thinking of a continuous family groupoid is that it is a groupoid $G$ such that $G_x$ is a smooth manifold for every $x\in  G^0$ and such that the smooth structure varies continuously in $x$ and such that right multiplication by $\gamma\in G$ gives a diffeomorphism $G_{r(\gamma)}\to G_{s(\gamma)}$ which also varies continuously in $\gamma$. 
 
 Given a continuous family groupoid $G$, we define $C^{\infty,0}(G)$ to be the space of continuous functions $f:G\to \C$ such that for every chart $(\Omega,\phi)$, $f\circ \phi^{-1}$ is $C^{\infty,0}$.
 
  We will show that $\mathbb{G}$ can be naturally equipped with charts so that it becomes a continuous family groupoid. Instead of defining charts on $\mathbb{G}$, we will define the space $C^{\infty,0}(\mathbb{G})$ and then show that it comes from charts.
\begin{dfn}\label{dfn:Cinfinity0} Let $f:\mathbb{G}\to \C$ be a continuous function. We say that $f\in C^{\infty,0}(\mathbb{G})$ if \begin{itemize}
\item The restriction of $f$ to $(y,x,t)\in M\times M\times \R_{+}^\times$ is $C^{\infty,0}$ where it is smooth in the variable $y$ and continuous in $x$ and $t$.
\item For every (or for some) graded basis $(V,\natural)$, the map $f\circ \cQ_V:V\times \mathbb{G}^0\to \C$ is $C^{\infty,0}$, i.e., it is smooth in $V$ and continuous in $\mathbb{G}^0$.
\end{itemize}
\end{dfn}
The fact that (for every) is equivalent to (for some) in Definition \ref{dfn:Cinfinity0} follows from the fact that the map $\hat{\phi}$ defined in \eqref{eqn:change_of_basis} is $C^{\infty,0}$ because $\phi$ is smooth.
\begin{theorem}\label{thm:cont_fam} There exists a unique continuous family groupoid structure on $\mathbb{G}$ such that its $C^{\infty,0}$ functions are the ones given by Definition  \ref{dfn:Cinfinity0}.
\end{theorem}
\begin{proof}
 Let $(\mathfrak{g},\natural)$ be a graded Lie basis, $x\in M$, $L\in \cG^0_x$, $g\in \mathfrak{g}$, $S\subseteq \mathfrak{g}$ a linear subspace which is complementary to $\natural_{x,0}^{-1}(L)$. We denote by $gS$ the submanifold of $\mathfrak{g}$ of elements of the form $gs$ for $s\in S$. The following lemma gives us charts around the point $(\natural_{x,0}(g)L,x,0)\in \cG\times \{0\}$.
\begin{lem}\label{lem:qjsdflioqjspfd} There exists a $C^{\infty,0}$ map $\psi:\mathfrak{g}\times \mathbb{G}^0\to gS\times \mathbb{G}^0$ defined in a neighbourhood of $(g,L,x,0)$ such that the following diagram commutes \begin{equation}\label{eqn:Diagram_Cont_Fam_Lemma}
 \begin{tikzcd}[column sep = huge]
gS\times \mathbb{G}^0\arrow[r,"\cQ_{\mathfrak{g}|gS\times \mathbb{G}^0}"]&\mathbb{G}\\  \mathfrak{g}\times \mathbb{G}^0\arrow[ur,"\cQ_\mathfrak{g}"']\arrow[u,"\psi"].
\end{tikzcd}
\end{equation}
\end{lem}
\begin{proof}
Let $\phi$ be as in Theorem \ref{thm:composition_bisub}. Let $E$ be the total space of the tautological vector bundle over the space $\Grass(\mathfrak{g})\times M\times \R_+$ whose fiber over $(K,y,t)$ is $K$. Consider the map 
\begin{align*}
\kappa:gS\times E\to \mathfrak{g}\times\Grass(\mathfrak{g})\times M\times \R_+,\ \kappa(v,l,K,y,t)=(\phi(v,l,y,t),K,y,t),
\end{align*}
where $v\in gS,K\in \Grass(\mathfrak{g}), l\in K, (y,t)\in M\times \R_+$.
The differential of $\kappa$ is bijective at $(g,0,L,x,0)$. Hence, there exists $U_1\subseteq gS\times E$ and $U_2\subseteq \mathfrak{g}\times\Grass(\mathfrak{g})\times M\times \R_+$ neighbourhoods of $(g,0,L,x,0)$ and $(g,L,x,0)$ respectively such that $\kappa:U_1\to U_2$ is a diffeomorphism. We define $\psi$ whose domain is $U_2\cap (\mathfrak{g}\times \Grass(\natural)(\mathbb{G}^0))$ by $$\psi(v,\gamma)=(p_1(\kappa^{-1}(v,\Grass(\natural)(\gamma))),\gamma)\quad v\in \mathfrak{g},\gamma\in \mathbb{G}^0,$$ where $p_1:gS\times E\to gS$ is the natural projection, and $\Grass(\natural)$ is the map \eqref{eqn:Grass_natural}.  It is straightforward to check that $\psi$ has the required properties.
\end{proof}
By an argument, like in the beginning of Section \ref{sec:proof_Hasud}, we can show that one can find an open neighbourhood $\Omega\subseteq gS\times \mathbb{G}^0$ of $(g,L,x,0) $ such that $\cQ_{\mathfrak{g}|\Omega}$ is injective. We can further suppose that $\Omega\subseteq \im(\psi)$. It follows that $\cQ_{\mathfrak{g}|\Omega}:\Omega\to \mathbb{G}$ is an open embedding, which by Lemma \ref{lem:qjsdflioqjspfd} has the property that if $f\in C_c(\cQ_{\mathfrak{g}}(\Omega))$, then $f\in C^{\infty,0}(\mathbb{G})$ if and only if $f\circ\cQ_{\mathfrak{g}|\Omega} $ is $C^{\infty,0}$. By taking all charts of $\mathbb{G}$ compatible with the charts we just constructed we obtain charts of $\mathbb{G}$ which satisfy Condition 3,4,6 of Definition \ref{dfn:cont_fam}. To check Condition 5 for the charts $(\cQ_{\mathfrak{g}}(\Omega), \cQ_{\mathfrak{g}|\Omega}^{-1})$, we use the fact that the map $\hat{\phi}$ from \eqref{eqn:m_product} is $C^{\infty,0}$ because $\phi$ is $C^\infty$.
\end{proof}
\section{Lie algebroid}\label{sec:Lie_Alg}
Let $G$ be a continuous family groupoid. Like for Lie groupoids, we denote by $T_sG$ the topological vector bundle over $G$ whose fiber at $\gamma\in G$ is $T_{\gamma}G_{s(\gamma)}$. 
The vector bundle admits a natural $C^{\infty,0}$ structure, i.e., we have local charts for $T_sG$, which come from the charts of $G$ such that the transition functions are $C^{\infty,0}$. 
We denote by $C^{\infty,0}(G,T_sG)$ the space of $C^{\infty,0}$-sections of $T_sG$. We denote by $A(G)$ the restriction of $T_sG$ to $G^0$. Since right multiplication by $\gamma$ is a diffeomorphism, $G_{r(\gamma)}\to G_{s(\gamma)}$, 
it follows that  $T_sG$ is isomorphic (as a topological vector bundle) to $r^*A(G)$. Like for Lie groupoids, if $X\in C(G^0,A(G))$ is a continuous section, then $$\tilde{X}:=X\circ r\in C(G,r^*A(G))=C(G,T_sG)$$ is called the right invariant section of $T_sG$ associated to $X$.
\begin{dfn}\label{dfn:smooth_sections}
A continuous section $X\in C(G^0,A(\mathbb{G}))$ is called $C^{\infty,0}$ if the associated right-invariant section $\tilde{X}$ is $C^{\infty,0}$. We denote the space of $C^{\infty,0}$ sections by $C^{\infty,0}(G^0,A(\mathbb{G}))$.
\end{dfn}
If $X,Y\in C^{\infty,0}(G^0,A(\mathbb{G}))$, then $[\tilde{X},\tilde{Y}]$ is well-defined on each fiber, and it belongs to $C^{\infty,0}(G,T_sG)$. It is also right-invariant. Hence, it is equal to $\tilde{Z}$ for some $Z\in C^{\infty,0}(G^0,A(\mathbb{G}))$. Hence, we have a Lie bracket $$[\cdot,\cdot]:C^{\infty,0}(G^0,A(\mathbb{G}))\times C^{\infty,0}(G^0,A(\mathbb{G}))\to C^{\infty,0}(G^0,A(\mathbb{G})).$$
For a general continuous family groupoid, it is an open question in {\cite{ContFamilyGroupoids}} if there exists any $C^{\infty,0}$ section.

We now describe $A(\mathbb{G})$ and show that $A(\mathbb{G})$ admits many $C^{\infty,0}$ sections. The fiber of $A(\mathbb{G})$ over $(x,t)$ is equal to $T_xM$ and over $(L,x,0)$ is equal to $\frac{\gr(\cF)_x}{L}$, the quotient vector space of $\gr(\cF)_x$ by $L$. In local coordinates, $A(\mathbb{G})$ has a simple description as follows. Let $(V,\natural)$ be a graded basis. Then by \eqref{eqn:Grass_natural}, one has a closed embedding $$\Grass(\natural):\mathbb{G}^0\to \Grass(V)\times M\times \R_+.$$ On the space $\Grass(V)\times M\times \R_+$ there is a natural vector bundle $E$ whose fiber over $(L,x,t)$ is $V/L$. The restriction of $E$ to $\Grass(\natural)(\mathbb{G}^0)$ is $A(\mathbb{G})$. 
\begin{prop}\label{prop:cont_sec_alg}
Let $X\in \cF^i$. The section $\theta_i(X)$ of $A(\mathbb{G})$ defined by \begin{align*}
 \theta_i(X)(x,t)=t^iX(x)\in A(\mathbb{G})_{(x,t)}=T_xM,\ \theta_i(X)(L,x,0)=[X]_{i,x} \mod L
\end{align*}is a $C^{\infty,0}$ section.%I changed here for DUKE
\end{prop}
Before we prove Proposition \ref{prop:cont_sec_alg}, let us describe the right invariant section of $T_s\mathbb{G}$ associated to $\theta_i(X)$, denoted $\tilde{\theta}_i(X)$. If $(x,t)\in M\times \R_+^\times$, then the restriction of $\tilde{\theta}_i(X)$ to $\mathbb{G}_{(x,t)}=M$ is equal to $t^iX$. Let $(L,x,0)\in \mathbb{G}^0$. The restriction of $\tilde{\theta}_i(X)$ to $\mathbb{G}_{(L,x,0)}=\Gr(\cF)_x/L$ has a simple description as follows. The group $\Gr(\cF)_x$ acts on the left on $\Gr(\cF)_x/L$. The derivative of the action gives a Lie algebra morphism $\rho:\gr(\cF)_x\to \cX(\Gr(\cF)_x/L)$. The restriction $\tilde{\theta}_i(X)$ to $\mathbb{G}_{(L,x,0)}$ is $\rho([X]_{i,x})$.
\begin{proof}
It is clear that $\tilde{\theta}_i(X)$ is $C^{\infty,0}$ on $M\times M\times \R_+^\times$. Let $(\mathfrak{g},\natural)$ be a graded Lie basis with $v\in \mathfrak{g}^i$ such that $\natural(v)=X$. It is enough to prove that $\tilde{\theta}_i(X)\circ \cQ_{\mathfrak{g}}$ is a $C^{\infty,0}$ section of $\cQ_{\mathfrak{g}}^*(T_sG)$. The map $\cQ_\mathfrak{g}$ is $C^{\infty,0}$, so if $Y:\mathfrak{g}\times \mathbb{G}^0\to \mathfrak{g}$ is a partially defined $C^{\infty,0}$ section, then $dQ_{\mathfrak{g}}(Y)$ is a $C^{\infty,0}$ section of $\cQ_{\mathfrak{g}}^*(T_sG)$. Hence, it suffices to construct $Y$ such that $dQ_{\mathfrak{g}}(Y)=\tilde{\theta}_i(X)\circ \cQ_{\mathfrak{g}}$.

 Let $\phi$ be as in Theorem \ref{thm:composition_bisub}. We define the partially defined $Y:\mathfrak{g}\times \Grass(\mathfrak{g})\times M\times \R_+\to \mathfrak{g}$ by $$Y(g,L,x,t)=\frac{d}{ds}\Bigr|_{\substack{s=0}}\phi(sv,g,x,t).$$
 Clearly $Y$ is smooth on its domain of definition. We now restrict $Y$ to $\mathfrak{g}\times \Grass(\natural)(\mathbb{G}^0)$. Hence, it becomes a $C^{\infty,0}$ partially defined function $Y:\mathfrak{g}\times \mathbb{G}^0\to \mathfrak{g}$. Finally, the fact that $d\cQ_{\mathfrak{g}}(Y)=\tilde{\theta}_i(X)\circ \cQ_{\mathfrak{g}}$ follows trivially from Theorem \ref{thm:composition_bisub}.2 and \ref{thm:composition_bisub}.3.
\end{proof}
 Let $X\in \cF^i$, $Y\in \cF^j$. It is clear that \begin{equation}\label{eqn:qhsuifdilq}
 \tilde{\theta}_{i+j}([X,Y])=[\tilde{\theta}_i(X),\tilde{\theta}_i(Y)].
\end{equation}Let $x\in M$ and $L\in \cG^0_x$. We denote by $\cF^i/L\subseteq \cX(\Gr(\cF_x)/L)$ the module generated by the restrictions of $\tilde{\theta}_i(X)$ to $\Gr(\cF)_x/L$ as $X$ varies in $\cF^i$. The module $\cF^i/L$ is finitely generated, and $\cF^N/L=\cX(\Gr(\cF)_x/L)$. By \eqref{eqn:qhsuifdilq}, one has $$[\cF^i/L,\cF^j/L]\subseteq \cF^{i+j}/L.$$ Hence, the manifold $\Gr(\cF_x)/L$ is equipped with a weighted sub-Riemannian structure.
\begin{rem}On $\Gr(\cF)_x/\mathfrak{r}_x$, the weighted sub-Riemannian structure $\cF^i/\mathfrak{r}_x$ agrees with the one defined by Bellaïche \cite[Definition 5.15]{BellaicheArt} in the sub-Riemannian case. 
\end{rem}
\section{Debord-Skandalis action}\label{sec:Debord-Skand-Act}
The group $\R_+^\times$ acts on all the objects we have constructed in this article as follows. Let $\lambda\in \R_+^\times$. For every $x\in M$, we define a Lie algebra homomorphism $$\alpha_\lambda:\gr(\cF)_x\to \gr(\cF)_x,\quad \alpha_\lambda\left(\sum_{i=1}^Nv_i\right)=\sum_{i=1}^N\lambda^iv_i,\quad v_i\in \frac{\cF^i}{\cF^{i-1}+I_x\cF^i}.$$
The map $\alpha_\lambda$ is also a Lie group homomorphism. We also define 
$$\alpha_\lambda:\mathbb{G}^0\to \mathbb{G}^0,\quad \alpha_\lambda(x,t)=(x,\lambda^{-1}t),\quad \alpha_\lambda(L,x,0)=(\alpha_\lambda(L),x,0)$$

\begin{prop}
The map $\alpha_\lambda:\mathbb{G}^0\to \mathbb{G}^0$ is well-defined, i.e., if $L\in \cG^0_x$, then $\alpha_\lambda(L)\in \cG^0_x$. Furthermore, $\alpha_\lambda:\mathbb{G}^0\to \mathbb{G}^0$ is a homeomorphism.
\end{prop}
\begin{proof}
Let $(V,\natural)$ be a graded basis. We define $$\alpha_\lambda:V\to V,\quad \alpha_\lambda\left(\sum_{i=1}^Nv_i\right)=\sum_{i=1}^N\lambda^iv_i,\quad v_i\in V^i.$$We extend this map to $\Grass(V)\times M\times \R_+$ by  
$$\alpha_\lambda:\Grass(V)\times M\times \R_+\to \Grass(V)\times M\times \R_+, \alpha_\lambda(L,x,t)=(\alpha_\lambda(L),x,\lambda^{-1}t).$$The map $\alpha_\lambda:\Grass(V)\times M\times \R_+\to \Grass(V)\times M\times \R_+$ is clearly a homeomorphism. The following diagram commutes 
$$\begin{tikzcd}[column sep=large]\mathbb{G}^0\arrow[r,"\Grass(\natural)"]\arrow[d,"\alpha_\lambda"]&\Grass(V)\times M\times \R_+\arrow[d,"\alpha_\lambda"]\\\mathbb{G}^0\arrow[r,"\Grass(\natural)"]&\Grass(V)\times M\times \R_+
\end{tikzcd}$$
By Proposition \ref{thm:top_well_defined}.1, the result follows.
\end{proof}
We also let $\R_+^\times$ act on $\mathbb{G}$ by $$\alpha_\lambda:\mathbb{G}\to \mathbb{G},\quad \alpha_\lambda(y,x,t)=(y,x,\lambda^{-1}t),\quad \alpha_\lambda(gL,x,0)=(\alpha_\lambda(gL),x,0).$$

\begin{prop}
The map $\alpha_\lambda:\mathbb{G}\to \mathbb{G}$ is a homeomorphism.
\end{prop}
\begin{proof}
Let \begin{equation}\label{eqn:action_VGO}
 \alpha_\lambda:V\times \mathbb{G}^0\to V\times \mathbb{G}^0,\quad \alpha_\lambda(v,\gamma)=(\alpha_\lambda(v),\alpha_\lambda(\gamma)),\quad v\in V,\gamma\in \mathbb{G}^0.
\end{equation} One easily verifies that the diagram 
$$\begin{tikzcd} V\times\mathbb{G}^0\arrow[r,"\cQ_V"]\arrow[d,"\alpha_\lambda"]&\mathbb{G}\arrow[d,"\alpha_\lambda"]\\ V\times\mathbb{G}^0\arrow[r,"\cQ_V"]&\mathbb{G}
\end{tikzcd}$$ commutes. The result follows.
\end{proof}
It is straightforward to check that the map $\alpha_\lambda:\mathbb{G}\to \mathbb{G}$ is a Lie groupoid automorphism.
\begin{dfn}
Let $(V,\natural)$ be a graded basis. A quasi-norm is a continuous function  $\norml{\cdot}:V\to \R_+$ such that \begin{enumerate}
\item $\norml{\alpha_\lambda(v)}=\lambda \norml{v}$ for all $\lambda\in \R_+^\times$ and $v\in V$.
\item $\norml{v}=0$ if and only if $v=0$.
\end{enumerate} 
\end{dfn}
Let $(V,\natural)$ be a graded basis and $\norml{\cdot}$ a quasi-norm. If $(y,x,t)\in M\times M\in \R_+^\times$, then we define \begin{equation}\label{eqn:1Norml}
 \norml{(y,x,t)}:=\inf \{\norml{v}:y=\exp(\natural_t(v))\cdot x\}.
\end{equation} There is no reason in general for $y=\exp(\natural_t(v))\cdot x$ to admit a solution in $V$, so $\norml{(y,x,t)}$ can be infinite. It is finite if $y$ and $x$ are close to each other. If $(gL,x,0)\in \cG\times \{0\}$, we also define 
\begin{equation}\label{eqn:2Norml}
\norml{(gL,x,0)}=\inf\{\norml{v}:v\in \natural_{x,0}^{-1}(gL)\}.\end{equation} 
We have thus defined a map $$\norml{\cdot}:\mathbb{G}\to [0,+\infty]$$ which is $1$-homogeneous, i.e., $$\norml{\alpha_\lambda(\gamma)}=\lambda\norml{\gamma},\quad \lambda\in \Rpt,\gamma\in \mathbb{G}.$$
We remark that if $\gamma\in \mathbb{G}$, then $\norml{\gamma}=0$ if and only if $\gamma\in \mathbb{G}^0$.

There is no reason for the map $\norml{\cdot}$ to be continuous on $\mathbb{G}$. Nevertheless, we prove that it is continuous near $\mathbb{G}^0$.
\begin{theorem}\label{thm:cont_norm}
Let $\mathbb{U}$ as in Proposition \ref{rem:cQ_V_quotient}. Then the map $\norml{\cdot}:\cQ_V(\mathbb{U})\to [0,+\infty[$ is continuous.
\end{theorem}
\begin{proof}
The infimum in \eqref{eqn:1Norml} and \eqref{eqn:2Norml} is attained. One can then easily deduce the continuity of $\norml{\cdot}$ on $\mathbb{U}\cap (M\times M\times \Rpt)$ from Proposition \ref{rem:cQ_V_quotient}.2. Let $(y_n,x_n,t_n)\in M\times M\times \Rpt$ a sequence converging to $(gL,x,0)$. We will show that $\norml{(y_n,x_n,t_n)}\to \norml{(gL,x,0)}$. Let $v_n\in V$ such that $y_n=\exp(\natural_{t_n}(v_n))\cdot x_n$ and $\norml{v_n}=\norml{(y_n,x_n,t_n)}$. Let $w\in gL$ such that $\norml{w}=\norml{(gL,x,0)}$. By Proposition \ref{rem:cQ_V_quotient}, we deduce that there exists $w_n\in V$ which converge to $w\in gL$ such that $y_n=\exp(\natural_{t_n}(w_n))\cdot x_n$. Hence, $\norml{v_n}\leq \norml{w_n}$ and $\norml{v_n}$ is bounded. By passing to a subsequence, we can suppose that $v_n\to v$. Then $\norml{v}\leq \norml{w}$. Again using Proposition \ref{rem:cQ_V_quotient}, we deduce that $v\in gL$. Hence, $\norml{v}=\norml{w}$. So $\norml{(y_n,x_n,t_n)}\to \norml{(gL,x,0)}$. It remains to show that if $(g_nL_n,x_n,0)$ converge to $(gL,x,0)$, then $\norml{(g_nL_n,x_n,0)}\to \norml{(gL,x,0)}$. This is proven like before and is left to the reader.
\end{proof}
\section{Blowup space using the sub-Riemannian metric}\label{sec:blowup_Riem}		
We denote by $\mathbb{P}(\mathbb{G})$ the quotient of $\mathbb{G}\backslash \mathbb{G}^0$ by the $\R_+^\times$-action equipped with the quotient topology. Let $\mathbb{P}(\cG)$ be the quotient of $\cG\backslash \cG^0$ by the $\R_+^\times$-action. 
Then one has $$\mathbb{P}(\mathbb{G})=(M\times M\backslash \Delta_M)\sqcup \mathbb{P}(\cG),$$ where $\Delta_M\subseteq M\times M$ is the diagonal. One can think of $\mathbb{P}(\mathbb{G})$ as a kind of topological blowup of $M\times M$ along the diagonal which 
respects the weighted sub-Riemannian structure. There is a natural map $\pi:\mathbb{P}(\mathbb{G})\to M\times M$ which sends the class of $(y,x,t)$ to $(y,x)$ and the class of $(gL,x,0)$ to $(x,x)$.
\begin{theorem}\label{thm:blwup}
The space $\mathbb{P}(\mathbb{G})$ is a second countable locally compact metrizable space. Furthermore, the map $\pi$ is a continuous surjective proper map.
\end{theorem}
\begin{proof}We will show that the $\R_+^\times$-action on $\mathbb{G}\backslash \mathbb{G}^0$ is proper. 
The $\R_+^\times$-action on $\cG\backslash \cG^0$ is easily seen to be proper. Take a sequence $\lambda_n\in \Rpt$ with $\lambda_n\to 0$ or $+\infty$, $(y_n,x_n,t_n)\in M\times M\times \Rpt$ with $y_n\neq x_n$ and $(y_n,x_n,t_n)$ converging in $\mathbb{G}\backslash \mathbb{G}^0$. We will show that $\alpha_{\lambda_n}(y_n,x_n,t_n)$ has to diverge to $\infty$ in $\mathbb{G}\backslash \mathbb{G}^0$. Let $(V,\natural)$ be a graded basis. 
\begin{itemize}
\item If $(y_n,x_n,t_n)\to (y,x,t)$ with $t>0$ and $x\neq y$ and $\lambda_n\to 0$ or $+\infty$, then $\alpha_{\lambda_n}(y_n,x_n,t_n)$ is easily seen to diverge to $\infty$ in $\mathbb{G}$.
\item If $(y_n,x_n,t_n)\to (gL,x,0)$ with $gL\neq L$. Then suppose that $$\alpha_{\lambda_n}(y_n,x_n,t_n)=(y_n,x_n,t_n\lambda_n^{-1})$$ converges in $\mathbb{G}\backslash \mathbb{G}^0$. Since $y_n,x_n\to x$, the limit of $(y_n,x_n,t_n\lambda_n^{-1})$ has to be of the form $(hK,x,0)$ for some $hK\in \cG^0_x$ with $hK\neq K$. Suppose first $\lambda_n\to +\infty$. Then by Proposition \ref{prop:Conv}, we can write $y_n=\exp(\natural(\alpha_{t_n\lambda_n^{-1}}(v_n)))\cdot x_n$ with $v_n\to v\in hK$. Then take $w_n=\alpha_{\lambda_n^{-1}}(v_n)$. Since $y_n=\exp(\natural(\alpha_{t_n}(w_n)))\cdot x$ and $w_n\to 0$, it follows from Proposition \ref{prop:Conv} that $(y_n,x_n,t_n)\to (L,x,0)$ which is a contradiction. If $\lambda_n\to 0$. Then by Proposition \ref{prop:Conv}, we can write $y_n=\exp(\natural(\alpha_{t_n}(v_n))\cdot x_n$ with $v_n\to v\in gL$.  Then take $w_n=\alpha_{\lambda_n}(v_n)$. We can repeat the previous argument to deduce that $(y_n,x_n,t_n\lambda_n^{-1})$ converges to $(K,x,0)$, again a contradiction.
\end{itemize}
This shows that the action of $\R_+^\times$ on $\mathbb{G}\backslash \mathbb{G}^0$ is proper. Hence, $\mathbb{P}(\mathbb{G})$ is second countable locally compact Hausdorff. We now show that $\pi$ is proper. We fix a quasi-norm on $V$. We have two cases to consider \begin{itemize}
\item If $(y_n,x_n)\in M\times M\backslash \Delta_M$, with $y_n,x_n\to x$. We need to show that there exists a subsequence of $(y_n,x_n)$ which converges in $\mathbb{P}(\mathbb{G})$. Since $y_n,x_n$ converge to the same point, there exists $v_n\in V$ such that $y_n=\exp(\natural(v_n))\cdot x_n$, and $\norml{v_n}=\norml{(y_n,x_n,1)}$. Let $t_n=\norml{v_n}$. Since $y_n\neq x_n$, it follows that $t_n\in \Rpt$. Since $y_n,x_n\to x$, it follows that $t_n\to 0$. By taking a subsequence, we can suppose that $\alpha_{t_n^{-1}}(v_n)$ converges to $v\in V$, and $(x_n,t_n)$ converges to $(L,x,0)$ for some $L\in \mathbb{G}^0$. Then by Proposition \ref{rem:cQ_V_quotient}, $(y_n,x_n,t_n)$ converges to $(\natural_{x,0}(v)L,x,0)$. By Theorem \ref{thm:cont_norm}, one has $\norml{(y_n,x_n,t_n)}\to \norml{(\natural_{x,0}(v)L,x,0)}$. By homogeneity, we have $$\norml{(y_n,x_n,t_n)}=t_n^{-1}\norml{(y_n,x_n,1)}=t_n^{-1}\norml{v_n}=1.$$ Hence, $\norml{(\natural_{x,0}(v)L,x,0)}=1$. So $\natural_{x,0}(v)L\neq L$. Hence, the class of $(y_n,x_n)$ in $\mathbb{P}(\mathbb{G})$ converges to the class of $(\natural_{x,0}(v)L,x,0)$ in $\mathbb{P}(\mathbb{G})$.
\item If $(g_nL_n,x_n,0)$ is a sequence in $\mathbb{G}\backslash \mathbb{G}^0$ with $x_n\to x$. We need to show that there exists a subsequence of $(g_nL_n,x_n,0)$ which converges in $\mathbb{P}(\mathbb{G})$. By replacing $(g_nL_n,x_n,0)$ with $\alpha_{\lambda}(g_nL_n,x,0)=(\alpha_\lambda(g_n)\alpha_\lambda(L_n),x,0)$, we can suppose that $\norml{(g_nL_n,x_n,0)}=1=\norml{v_n}$ for some $v_n\in \natural_{x_n,0}^{-1}(g_nL_n)$. By taking a subsequence, we can suppose that $v_n\to v$ for some $v\in V$ and $(L_n,x_n,0)\to (L,x,0)$ in $\mathbb{G}^0$. Then $(g_nL_n,x_n,0)\to (\natural_{x,0}(v)L,x,0)$ by Proposition \ref{prop:Conv}. Furthermore, $\norml{(\natural_{x,0}(v)L,x,0)}=1$ by Theorem \ref{thm:cont_norm}. Hence, $\natural_{x,0}(v)L\neq L$ and so $(\natural_{x,0}(v)L,x,0)$ defines an element of $\mathbb{P}(\mathbb{G})$.\qedhere
\end{itemize}
\end{proof}

\section{Carnot-Carathéodory metric and Gromov-Hausdorff convergence}\label{sec:TangentCones}
In this section, we suppose that we have a sub-Riemannian structure on $M$, i.e., \begin{equation}\label{eqn:sub-riem}
 \cF^{i+1}=\cF^{i}+[\cF^{i},\cF^{1}]
\end{equation} for all $i\geq 1$. 

We recall the following notion. If $V,W$ are finite dimensional vector spaces, $\norm{\cdot}$ is a norm on $V$, $f:V\to W$ is a surjective linear map, then the following formula $$\norm{w}:=\inf_{v\in f^{-1}(w)}\norm{v}$$
defines a norm on $W$ which we call the quotient norm by $f$.

 Let $(V,\natural)$ be a graded basis. We equip $V^1$ with a Euclidean norm $\norm{\cdot}$. For $x\in M,t\neq 0$, we denote by $\norm{\cdot}_{x,t}$ the quotient norm on $\ev_x(\cF^1)$ by $$\natural_{x,t|V^1	}:V^1\to \ev_x(\cF^1)=\{X(x):X\in \cF^1\}.$$ It is clear that if $v\in \ev_x(\cF^1)$, then $\norm{v}_{x,t}=\frac{1}{t}\norm{v}_{x,1}$. 
 
 If $c:I\to M$ is an absolutely continuous path such that $c'(s)\in \ev_{c(s)}(\cF^1)$ for almost all $s\in I$, then we define $$ \mathrm{length}_t(c):=\int_{0}^1\norm{c'(s)}_{c(s),t}ds.$$We also define a distance by\begin{equation*}
d_t(x,y)=\inf \mathrm{length}_t(c)
\end{equation*}
where the infimum is taken over all absolutely continuous paths from $x$ to $y$. It is clear that  \begin{equation}\label{eqn:dist}
d_t(x,y)=\frac{d_1(x,y)}{t}.
\end{equation}

Let $(L,x,0)\in \mathbb{G}$. Notice that \eqref{eqn:sub-riem} implies that $$\cF^{i+1}/L=\cF^{i}/L+[\cF^{i}/L,\cF^{1}/L].$$Hence, $\cF^i/L$ also defines a sub-Riemannian structure on $\Gr(\cF)_x/L$. The derivative of the left $\Gr(\cF)_x$-action on $\Gr(\cF)_x/L$, gives a linear map $\gr(\cF)_x\to \cX(\Gr(\cF)_x/L)$. By composing this map with $\natural_{x,0}:V\to \gr(\cF)_x$, we get a map $V\to \cX(\Gr(\cF)_x/L)$. The image of $V^1$ is a family of generators for $\cF^1/L$. Hence, for any $gL\in \Gr(\cF)_x/L$, we have a surjective linear map \begin{equation}\label{eqn:qshduifiq}
 V^1\to  \cF^1/L\to \ev_{gL}( \cF^1/L)=\{X(gL)\in T_{gL}\Gr(\cF)_x/L:X\in\cF^1/L \}.
\end{equation}
The norm $\norm{\cdot}$ on $V^1$, by \eqref{eqn:qshduifiq}, defines a quotient norm on $\ev_{gL}( \cF^1/L)$, denoted by $\norm{\cdot}_{\cG(\cF)_x/L}$. If $c:I\to \Gr(\cF)_x/L$ is an absolutely continuous path with $c'(s)\in \ev_{c(s)}( \cF^1/L)$ for almost all $s\in I$, then we define $$ \mathrm{length}_t(c):=\int_{0}^1\norm{c'(s)}_{\cG(\cF)_x/L}ds.$$We also define a distance by \begin{equation*}
d_{\Gr(\cF)_x/L}(gL,hL)=\inf \mathrm{length}_t(c),
\end{equation*}
where the infimum is taken over all absolutely continuous paths from $gL$ to $hL$. We remark that the way we defined the distance $d_{\Gr(\cF)_x/L}$ implies immediately that if $hL\in \Gr(\cF)_x/L$, then right multiplication by $Lh^{-1}$ is an isometry 
$\Gr(\cF)_x/L\to \Gr(\cF)_x/hLh^{-1}$. Hence, \begin{equation}\label{eqn:equiv_dist}
d_{\Gr(\cF)_x/L}(gL,hL)=d_{\Gr(\cF)_x/hLh^{-1}}(gLh^{-1},hLh^{-1}),
\end{equation}
where $L\in \cG^0_x,gL,hL\in \Gr(\cF)_x/L$.

\begin{theorem}\label{thm:cont_dist}
The function \begin{align*}
d_{\mathbb{G}}&:\mathbb{G}\to \R_+\\
d_{\mathbb{G}}(y,x,t)&=d_t(y,x)\quad (y,x,t)\in M\times M\times \R_+^\times\\
d_{\mathbb{G}}(gL,x,0)&=d_{\Gr(\cF)_x/L}(gL,L),\quad (gL,x,0)\in \cG\times\{0\}
\end{align*} is continuous, and homogeneous of degree $1$ with respect to the Debord-Skandalis action, i.e., $$d_\mathbb{G}(\alpha_\lambda(\gamma))=\lambda d_\mathbb{G}(\gamma),\quad  \forall\lambda\in \R_+^\times,\gamma\in \mathbb{G}.$$
\end{theorem}
\begin{proof}
Homogeneity follows from \eqref{eqn:dist} and the trivial to check fact that $\alpha_\lambda$ is an isometry from $(\Gr(\cF)_x/L,d_{\Gr(\cF)_x/L})$ to $(\Gr(\cF)_x/\alpha_\lambda(L),\lambda d_{\Gr(\cF)_x/\alpha_\lambda(L)})$. Continuity of $d_{\mathbb{G}}$ on $M\times M\times \R_+^\times$ follows from \cite[Corollary 2.6]{BellaicheArt}. By homogeneity, it is enough to prove continuity in a neighbourhood of $(L,x,0)$. The result follows from Theorem \ref{thm:cont_fam} and Proposition \ref{prop:cont_sec_alg} and \cite[Theorem 3.56]{ComprehensiveGuideSubRiemBook}.
\end{proof}
Let $\norml{\cdot}$ be a quasi-norm on $V$. Then $\norml{\cdot}:\mathbb{G}^0\to [0,+\infty]$ and $d_{\mathbb{G}}$ are $1$-homogeneous which vanish only on $\mathbb{G}^0$. Hence, the quotient $\frac{\norml{\cdot}}{d_{\mathbb{G}}}$ descends to a map $$\frac{\norml{\cdot}}{d_{\mathbb{G}}}:\mathbb{P}(\mathbb{G})\to ]0,+\infty].$$Let $\mathbb{U}$ be an open set like in Proposition \ref{rem:cQ_V_quotient}. Let $\mathbb{P}:\mathbb{G}\backslash \mathbb{G}^0\to \mathbb{P}(\mathbb{G}^0)$ be the quotient map. It is open because it is the quotient by a group action. Hence, $\mathbb{P}(\cQ_V(\mathbb{U})\backslash \mathbb{G}^0)$ is open. By Theorem \ref{thm:cont_norm}, we obtained an open neighbourhood of $\mathbb{P}(\cG)$ on which $\frac{\norml{\cdot}}{d_{\mathbb{G}}}$ is continuous. Since $\mathbb{P}(\cG)=\pi^{-1}(\Delta_M)$, where $\pi:\mathbb{P}(\mathbb{G})\to M\times M$ the map from Theorem \ref{thm:blwup}, we thus obtain the following \begin{theorem}\label{thm:comparison}
Let $x\in M$. Then there exists a compact neighbourhood $K\subseteq M\times M$ of $(x,x)$ and $C\in \Rpt$ such that on $\pi^{-1}(K)$, one has $$C^{-1}\leq \frac{\norml{\cdot}}{d_{\mathbb{G}}}\leq C$$
\end{theorem}
So if $x_0$ is fixed, and $y$ and $x$ are close enough to $x_0$, then one has the inequality $$C^{-1} d_{1}(y,x)\leq \inf\{\norml{v}\in V:y=\exp(\natural(v))\cdot x\}\leq Cd_1(y,x).$$ 

\begin{theorem}
Let $(x_n,t_n)\in M\times \R_+^\times$ be a sequence which converges in $\mathbb{G}^0$ to $(L,x,0)$. Then the pointed metric spaces $(M,d_{t_n},x_n)=(M,\frac{d_1}{t_n},x_n)$ converge in the pointed Gromov-Hausdorff sense to $(\Gr(\cF)_x/L,d_{\Gr(\cF)_x/L},L)$.
\end{theorem}
\begin{proof}
In the proof, if $(X,d)$ is a metric space, $x\in X$, $r\in \R_+$, then $B_X(x,r)$ or $B_d(x,r)$ denote the closed ball of radius $r$ and center $x$. 
\begin{lem}\label{lem:jqlhsifdjpoqsjdf} Let $(X_n,d_n)$, $(X,d)$ be metric spaces, $x_n\in X_n$. Suppose that $(X_n,d_n)$ are length metric spaces, and for each $n\in \N$, we have a map $\phi_n:X\to X_n$ such that \begin{enumerate}
\item $\phi_n$ is injective and $$\sup_{x,y\in X}|d(x,y)-d_n(\phi_n(x),\phi_n(y))|\xrightarrow{n\to +\infty} 0$$
\item There exists a sequence $R_n>0$ such that $\phi_n(X)\subseteq B_{X_n}(x_n,R_n)$ and $R_n\to R$ for some $R>0$.
\item For each $R'<R$, there exists $m$ such that for $n>m$, $ B_{X_n}(x_n,R')\subseteq \phi_n(X)$.
\end{enumerate}
Then $(B_{X_n}(x_n,R),d_n)$ converges in the Gromov-Hausdorff distance to $(X,d)$.
\end{lem}
Recall that a length metric space \cite[Definition 1.7 on Page 6]{GromovMetricStrucBook} is a metric space $(X,d)$ such that for any $x,y\in X$ and $\epsilon>0$, there exists a continuous path $l:[0,1]\to X$ such that $l(0)=x$ and $l(1)=y$ and the length of $l$ is less than or equal to $d(x,y)+\epsilon$. 
\begin{proof}
By replacing $R_n$ with $\max(R_n,R)$ we can suppose that $R_n\geq R$ for all $n$. Let $d_{GH}$ and $d_H$ denote the Gromov-Hausdorff and Hausdorff distance respectively. Condition $1$ implies that $$\lim_{n\to +\infty}d_{GH}((X,d),(\im(\phi_n),d_n))=0.$$Let $R'<R$. Take $n$ big enough so that Condition $3$ is satisfied. One has $$d_{H}(\im(\phi_n),B_{X_n}(x_n,R_n))\leq d_{H}(B_{X_n}(x_n,R'),B_{X_n}(x_n,R_n))\leq R_n-R'$$ where the last inequality follows from the fact that $X_n$ is a length space. 
Hence, $$\lim_{n\to +\infty}d_{GH}((\im(\phi_n),d_n),(B_{X_n}(x_n,R_n),d_n))=0.$$ Since $X_n$ is a length space, 
    \begin{equation*}\begin{aligned}
        \lim_{n\to +\infty}d_{GH}((B_{X_n}(x_n,R_n),d_n),(B_{X_n}(x_n,R),d_n))=0.
    \end{aligned}\end{equation*}
 The result follows from triangle inequality.
\end{proof}

 Let $(\Omega,\psi)$ be a chart of $\mathbb{G}$ such that $(L,x,0)\in \Omega$. Without loss of generality, we suppose that $\psi(\Omega)=U\times V$ where $U\subseteq \R^{\dim(M)}$ is an open subset and $V\subseteq \mathbb{G}^0$ is an open neighbourhood of $(L,x,0)\in \mathbb{G}^0$. We can further suppose that $(x_n,t_n)\in V$ for all $n\in \N$. We denote by $p:U\times V\to U$ the projection onto $U$. 

Now choose any $R>0$ small enough such that $B_{\Gr(\cF)_x/L}(L,R)\times \{(x,0)\}$ is a compact subset of $\Omega$. Let $\phi_n:B_{\Gr(\cF)_x/L}(L,R)\to M$ defined by $$(\phi_n(gL),x_n,t_n)=\psi^{-1}\Big(p(\psi(gL,x,0)),x_n,t_n\Big).$$By construction $(\phi_n(gL),x_n,t_n)\to (gL,x,0)$ for all $gL\in B_{\Gr(\cF)_x/L}(L,R)$. By Theorem \ref{thm:cont_structure_maps}, $$(\phi_n(gL),\phi_n(hL),t_n)\to (gLh^{-1},x,0),\quad \forall gL,hL\in B_{\Gr(\cF)_x/L}(L,R).$$By Theorem \ref{thm:cont_dist} and \eqref{eqn:equiv_dist}, it follows that \begin{equation}\label{eqn:limit_proof_ajhslia}
 d_{t_n}(\phi_n(gL),\phi_n(hL))\to d_{\Gr(\cF)_x/L}(gL,hL),\quad \forall gL,hL\in B_{\Gr(\cF)_x/L}(L,R).
\end{equation}

The limit \eqref{eqn:limit_proof_ajhslia} is uniform in $gL,hL$. 
To see this, suppose that it is not. Since $B_{\Gr(\cF)_x/L}(L,R)$ is compact, we can construct subsequences $(g_nL)_{n\in \N},(h_nL)_{n\in \N}\subseteq B_{\Gr(\cF)_x/L}(L,R)$ such that $g_nL\to gL$ and $h_nL\to hL$ yet $ d_{t_n}(\phi_n(g_nL),\phi_n(h_nL))$ 
does not converge to $d_{\Gr(\cF)_x/L}(gL,hL)$. It is clear from the definition of $\phi_n$ that $(\phi_n(g_nL),x_n,t_n)\to (gL,x,0)$. The same holds for $hL$. Hence, by Theorem \ref{thm:cont_structure_maps}, Theorem \ref{thm:cont_dist} and \eqref{eqn:equiv_dist} $d_{t_n}(\phi_n(g_nL),\phi_n(h_nL))$ converges to $d_{\Gr(\cF)_x/L}(gL,hL)$ which is a contradiction.

Let $R_n=\sup\{d_{t_n}(\phi_n(gL),x_n):gL\in B_{\Gr(\cF)_x/L}(L,R)\}$. We have $$ d_{t_n}(\phi_n(gL),x_n)\to d_{\Gr(\cF)_x/L}(gL,L),\quad \forall gL\in B_{\Gr(\cF)_x/L}(L,R).$$
The limit is also uniform by an argument similar to before. Hence, $R_n\to R$. We now check the last condition of Lemma \ref{lem:jqlhsifdjpoqsjdf}. Let $R'<R$. Suppose that the last condition is not satisfied. Hence, there exists $y_n\in B_{d_{t_n}}(x_n,R')$ such that $y_n\notin \mathrm{Im}(\phi_n)$. Clearly $y_n\to x$. By Theorem \ref{thm:comparison} and Proposition \ref{prop:Conv}, there exists a subsequence of $(y_n,x_n,t_n)$ which converges to $(gL,x,0)\in \mathbb{G}$ for some $gL$. By Theorem \ref{thm:cont_dist}, $gL\in B_{\Gr(\cF)_x/L}(L,R') $. Hence, $(gL,x,0)\in \Omega$. So $(y_n,x_n,t_n)\in \Omega$ for $n$ big enough. Then $\psi(y_n,x_n,t_n)=(a_n,x_n,t_n)$ and $a_n\to a$ with $\psi(gL,x,0)=(a,L,x,0)$. Since $R'<R$, it follows that $\psi^{-1}(a_n,L,x,0)=(h_nL,x,0)$ with $h_nL\in  B_{\Gr(\cF)_x/L}(L,R)$ for $n$ big enough. Hence, $y_n=\phi_n(h_nL)$, which is a contradiction.

By Lemma \ref{lem:jqlhsifdjpoqsjdf}, we obtain that $B_{d_{t_n}}(x_n,R)$ converges to $B_{\Gr(\cF)_x/L}(L,R)$. In fact all the above works with $R'<R$. Since $\cG^0_x$ is compact, we deduce that there exists $R>0$ such that any sequence 
$(y_n,s_n)\in M\times \Rpt$ converging to an element $(K,x,0)\in \mathbb{G}^0$ and for any $R'\leq R$, one has  $$ \lim_{n\to +\infty}B_{d_{s_n}}(y_n,R')=B_{\Gr(\cF)_x/K}(K,R').$$By using the $\Rpt$ action on $\mathbb{G}^0$, we deduce that for all $R\in \Rpt$, 
\begin{equation*}
 \lim_{n\to +\infty}B_{d_{s_n}}(y_n,R)=B_{\Gr(\cF)_x/K}(K,R). 
\end{equation*}
Hence, 
    \begin{equation*}\begin{aligned}
        \lim_{n\to +\infty}(M,d_{s_n},y_n)=(\Gr(\cF)_x/K,d_{\Gr(\cF)_x/K},K).
    \end{aligned}\end{equation*}
    This finishes the proof.
\end{proof}
\begin{refcontext}[sorting=nyt]
\printbibliography
\end{refcontext}
{\footnotesize
(Omar Mohsen) Paris-Saclay University, Paris, France
\vskip-2pt e-mail: \texttt{omar.mohsen@universite-paris-saclay.fr}}
\end{document}